\documentclass[reqno,centertags,12pt]{amsart}

\usepackage{amsmath,amsthm,amscd,amssymb,latexsym,verbatim}
\usepackage{tikz}
\usetikzlibrary{decorations.pathreplacing}
\usepackage{float}
\usepackage[colorlinks,pdfdisplaydoctitle,linkcolor=blue,citecolor=red]{hyperref}
\usepackage{caption}
\usepackage{subcaption}
\usepackage{url}
\usepackage{enumitem}
\usepackage{tikz}
\usepackage{color}
\usepackage{bbm}

\usepackage{graphicx,epsf,cite}

-\marginparwidth \oddsidemargin -4mm \evensidemargin \oddsidemargin

\newtheorem{theorem}{Theorem}[section]

\newtheorem{proposition}[theorem]{Proposition}
\newtheorem{lemma}[theorem]{Lemma}
\newtheorem{corollary}[theorem]{Corollary}
\theoremstyle{definition}
\newtheorem{definition}[theorem]{Definition}

\theoremstyle{remark}

\newcounter{smalllist}

\allowdisplaybreaks
\numberwithin{equation}{section}

\newcommand{\f}{\frac}

\newcommand{\beq}{\begin{equation}}
\newcommand{\eeq}{\end{equation}}

\newcommand{\bal}{\begin{align}}
\newcommand{\eal}{\end{align}}
\newcommand{\bals}{\begin{align*}}
\newcommand{\eals}{\end{align*}}

\newcommand{\del}{\delta}

\newcommand{\al}{\alpha}
\newcommand{\be}{\beta}
\newcommand{\ga}{\gamma}
\newcommand{\la}{\lambda}

\newcommand{\til}{\tilde}

\newcommand{\Om}{\Omega}

\newcommand{\mbr}{\mathbb R}

\newcommand{\mbn}{\mathbb N}

\newcommand{\Ga}{\Gamma}
\newcommand{\de}{\delta}
\newcommand{\De}{\Delta}

\newcommand{\ta}{\theta}
\newcommand{\si}{\sigma}
\newcommand{\Si}{\Sigma}
\newcommand{\cb}[1]{\left\{#1\right\}}
\newcommand{\sqb}[1]{\left[#1\right]}
\newcommand{\rb}[1]{\left(#1\right)}

\newcommand{\inn}[1]{\langle #1 \rangle}

\newcommand{\mclf}{\mathcal F}

\newcommand{\mcll}{\mathcal L}

\newcommand{\mbe}{\mathbb {E}}

\newcommand{\commentout}[1]{}

\begin{document}
\title[Propagation of Chaos in Stochastic Chemical Reaction-Diffusion Systems]
{Quantitative Propagation of Chaos in the Bimolecular Chemical Reaction-Diffusion model}

\author[T. Lim]{Tau Shean Lim}
\address[T. Lim]{Department of Mathematics, Duke University, Durham NC 27708, USA}
\email{taushean@math.duke.edu} 

\author[Y. Lu]{Yulong Lu}
\address[Y. Lu]{Department of Mathematics, Duke University, Durham NC 27708, USA}
\email{yulonglu@math.duke.edu} 

\author[J. Nolen]{James Nolen}
\address[J. Nolen]{Department of Mathematics, Duke University, Durham NC 27708, USA}
\email{nolen@math.duke.edu}



\begin{abstract}
We study a stochastic system of $N$ interacting particles which models bimolecular chemical reaction-diffusion.
In this model, each particle $i$ carries two attributes: the spatial location $X_t^i\in \mathbb{T}^d$,
and the type $\Xi_t^i\in \{1,\cdots,n\}$.
While $X_t^i$ is a standard (independent) diffusion process,
the evolution of the type $\Xi_t^i$ is described by pairwise interactions between different particles
under a series of chemical reactions described by a chemical reaction network.
We prove that, as $N \to \infty$, the stochastic system has a mean field limit which is described by a nonlocal reaction-diffusion partial differential equation.
In particular, we obtain a quantitative propagation of
chaos result for the interacting particle system. Our proof is based on
the relative entropy method used recently by Jabin and Wang  \cite{JW18}. The key ingredient of
the relative entropy method is a large deviation estimate for a special partition function, which was proved
previously by combinatorial estimates. We give a simple probabilistic proof based on a novel martingale argument.

\end{abstract}

\maketitle

\newcommand{\by}{{\boldsymbol y}}
\newcommand{\bx}{{\boldsymbol x}}
\newcommand{\bxi}{{\boldsymbol {\xi}}}
\newcommand{\mbx}{\mathbb{X}}
\newcommand{\bX}{\boldsymbol{X}}
\newcommand{\bY}{\boldsymbol{Y}}
\newcommand{\bXi}{\boldsymbol{\Xi}}
\newcommand{\fs}{\mathfrak{S}}
\newcommand{\fc}{\mathfrak{C}}
\newcommand{\fr}{\mathfrak{R}}
\newcommand{\fe}{\mathfrak{e}}
\newcommand{\mbo}{\mathbbm{1}}
\newcommand{\Ta} {\Theta}
\newcommand{\mch}{\mathcal{H}}
\newcommand{\bDe} {\boldsymbol{\De} }
\newcommand{\bS} {\boldsymbol{S} }
\newcommand{\bE} {\boldsymbol{E} }
\newcommand{\bL} {\boldsymbol{\mcll} }
\newcommand{\bT} {\boldsymbol{T} }

\long\def\metanote#1#2{{\color{#1}\
\ifmmode\hbox\fi{\sffamily\mdseries\upshape [#2]}\ }}
\long\def\TS#1{\metanote{green!55!black}{{\tiny TS} #1}}
\long\def\YL#1{\metanote{blue!70!black}{{\tiny YL} #1}}
\long\def\JN#1{\metanote{red!70!black}{{\tiny JN} #1}}



\section{Introduction}
In this paper, we consider a class of stochastic interacting particle systems modeling a chemical reaction-diffusion process.
In this model, there are $N$ particles, indexed by $i\in \{1,\cdots, N\}$. Each particle carries two attributes: a location $X_t^i\in \mathbb{T}^d$ (the $d$-dimensional torus),
and a chemical type $\Xi_t^i\in\{S_1,S_2,\cdots, S_n\}=\{1,\cdots, n_s\}$, where $n_s$ is the number of distinct chemical species.
As time $t$ progresses, $X_t^i$ diffuses as a standard Brownian motion in space independently, with the speed of diffusion depending on its type $\Xi_t^i$.
The type $\Xi_t^i$ changes in time according the pairwise chemical interactions between different particles, with transition rates depending on their locations,
types, and a set of \emph{bimolecular} chemical reactions $R_1,R_2,\cdots, R_{n_r}$ of the form
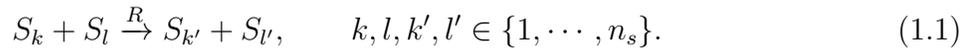
\begin{align}\label{bi-reaction}
S_k+S_l\xrightarrow{R} S_{k'}+S_{l'},\qquad k,l,k',l'\in \{1,\cdots,n_s\}.
\end{align}
Specifically, a pair of particles $i,j$ with types matching the input of a reaction $R$ (i.e. $(\Xi_t^i,\Xi_t^j)=(k,l)$ or $(l,k)$ 
in \eqref{bi-reaction}) may react and instantly change to types $(k',l')$ or $(l',k')$ at a random time that depends on their spatial location $X_t^i,X_t^j$. 
Our main result (Theorem \ref{main}) shows that in the limit $N \to \infty$, the empirical measure of the particles converges in a suitable sense to the solution of a nonlinear system of reaction-diffusion equations.

The precise description of the stochastic system will be given in the next section, using the notations from chemical reaction network theory.
For the time being, let us consider a simple special case, which involves only two types of particles $S_1,S_2$ (and hence $n_s=2$), and a single irreversible chemical reaction (hence $n_r=1$)
\begin{align}\label{special}
S_1+S_2\to 2S_2.
\end{align}
As mentioned earlier each particle diffuses independently in space with diffusivity depending on its type. A reaction, which turns a type-$S_1$ particle into a type-$S_2$ particle, happens at a random time, with a rate depending on its location relative to the type-$S_2$ particles. The reaction \eqref{special} is irreversible, in the sense that
once a particle turns into type-$S_2$, it can never turn back to type-$S_1$ again.
Hence, we have the following stochastic system:
\begin{align}\label{SRDE}
\left\{ \begin{array}{ll}
X^i_t&=X_0^i+ \displaystyle \int_0^t \si(\Xi_{s}^i)dB^i_s, \\
\Xi_t^i&=\Xi_0^i + \displaystyle E^i\rb{\f 1 N \sum_{j=1,j\neq i}^N \int_0^t \mbo\{(\Xi_{s-}^i,\Xi_{s-}^j)=(1,2)\} \Phi(X_s^i-X_s^j)ds },\\
&\qquad \mbox{ for }1\le i\le N.
\end{array} \right.
\end{align}
where $\si:\{1,2\}\to \mbr^+,\Phi\in (L^1\cap L^\infty)(\mathbb{T}^d)$ are prescribed diffusion coefficients and a reaction rate kernels, $\{B^i_t\}_{1\le i\le N}$ are independent standard $d$-dimensional Brownian motions, and $\{E^i(t)\}_{1\le i\le N}$ are independent unit Poisson jump processes.
Here, if the type $\Xi_t^i$ of a particle $i$ initially is 1 (hence a type-$S_1$ particle), it then turns into 2 at a rate of $\f 1N\sum_{j:\Xi_{t}^j=2}\Phi(X_t^i-X_t^j)$.
The scaling $\f 1N$ here is critical in deriving the mean field limit.
When the type $\Xi_t^i$ turns into 2, it remains a type-2 forever, because the rate is then zero (as $\mbo\{(\Xi_s^i,\Xi_s^j)=(1,2)\}=0$).

The objective of this paper is to establish the convergence, as $N\to\infty$, of the empirical measure of the process (which is a time-dependent random measure on $\mathbb{T}^d\times \{1,2,\cdots,n\}$)
\begin{align}\label{emp-meas}
\mu_N(t)=\f 1 N \sum_{i=1}^N \del_{(X_t^i,\Xi_t^i)}
\end{align}
to the solution $\bar \rho$ of a mean field limit equation, which is a deterministic system of $n$ non-local reaction-diffusion equations, provided that the initial distribution $\mu_N(0)$ converges to $\bar \rho(0)$ in some appropriate sense.
For instance, in the special case \eqref{special}, the limiting system of equations is
\begin{align}\label{RDS}
\left\{ \begin{array}{ll}
\partial_t u&=\displaystyle \f {\si(1)^2}2 \De u -(\Phi*w)u, \\
\partial_t w&=\displaystyle \f{\si(2)^2}2 \De w + (\Phi*w)u,
\end{array} \right.
\end{align}
where the components $u,w:[0,\infty)\times \mathbb{T}^d\to \mbr^+$ represent the distribution of type-1,2 particle respectively, and ``$*$" denotes the convolution operator
\begin{align*}
(\Phi*u)(x)=\int_{\mathbb{T}^d} \Phi(x-y)u(y)dy.
\end{align*}
The limit system in the general case, $n \geq 2$, is given below in \eqref{limitsystem}.

In the case that the initial distribution is well-mixed (constant density), the mean-field limit coincides with a \emph{mass action} system \cite{HJ72}.  In this case, the normalized concentrations $\{\bar \rho_k(t)\}_{k=1}^n$ for the $n$ chemical species do not depend on $x$, and they satisfy a system of $n$ ordinary differential equations. In the special case \eqref{special}, this is a simple Lotka-Volterra system, with $(u,w)=(\bar \rho_1,\bar \rho_2)$, and some $\la>0$:
\begin{align}
\left\{ \begin{array}{ll}
\dot u&= -\la uw,\\
\dot w&= \la uw,
\end{array} \right.
\end{align}
(cf. \eqref{RDS}).
There is also a stochastic counterpart of these systems in which case the total number $\{M_k(t)\}_{t\ge 0}$ of each species $k\in \{1,\cdots,n\}$ is a counting process satisfying a certain coupled stochastic system described by appropriate time-changed Poisson processes, whose rates depend on the current configuration $\{M_k(t)\}_{1\le k\le n}$.
Again, for the special case \eqref{special}, the stochastic system is given by (with $(U,W)=(M_1,M_2)$)
\begin{align*}
\left\{ \begin{array}{ll}
U_t&= U_0 -\displaystyle  E\rb{\la \int_0^t U_{s-}W_{s-}ds }, \\
W_t&= W_0 +\displaystyle  E\rb{\la \int_0^t U_{s-}W_{s-}ds } ,
\end{array} \right.
\end{align*}
where $E$ is a unit rate Poisson process (cf. \eqref{SRDE}).
With an appropriate scaling (in the rate of reactions), it can be shown that the process $\{M^N_k(t)\}_{1\le k\le n}$ described by the stochastic system converges, in a certain sense, to the solution $\{\bar \rho_k(t)\}_{1\le k\le n}$ of a mass action ODE system as $N\to\infty$ \cite{Kurtz71, AK15}.
This convergence result (at least for the case of bimolecular reactions), is a special case of our main result, when the initial distribution is well-mixed spatially (i.e., $X^i_t$, conditioned on $\Xi^i_t$, is a uniform random variable).

Models for chemical reactions with spatial diffusion have been studied for more than a century,
going back to the work of Smoluchowski \cite{Sm17} (see also \cite{Doi76a, AI14}).
Nevertheless, there are relatively few works that make a mathematically rigorous and
quantitative connection between stochastically interacting particles (as a microscopic model) 
and systems of reaction-diffusion equations, as a mean field limit -- this is the motivation for our work. 
De Massi, Ferrari, and Lebowitz \cite{DMFL86} derived a scalar reaction-diffusion equation as the limit 
of a Glauber-type spin system on a lattice $\mathbb{Z}^d$. This result was generalized by Durrett and Neuhauser 
\cite{DN94} to allow for more than two states/types, leading to a system of reaction-diffusion equations 
for reactions of the form $S_k \to S_j$, with rate depending on the density of other types.  
The reaction-diffusion limit is useful for studying phase transitions in the underlying stochastic model.
See \cite{CDP13} for related results in the spatially-discrete setting. 
In a continuum setting, Oelschl\"ager \cite{Oel89} analyzed a system of diffusing particles 
in which each particle may give birth $(S_k \to 2S_k)$, die ($S_k \to \emptyset$),
or change its type $(S_k \to S_j)$ at rates that depend on the density of other types,
leading to a system of reaction-diffusion equations in the infinite population limit. 
See \cite{Dit88, Kot86, Kot88, FLO19} for related works involving scalar reaction-diffusion equations.
Whether in the spatially discrete or continuous setting, most of these models involve only one particle changing
its type at a time, although the rate of the reaction may depend on the types and locations of other particles.
The recent paper \cite{ABRS19}
studies the 
mean field limit of a leader-follower dynamics, which models transitions between two labels (followers and leaders).
The mean field limit obtained there involves transport and reaction, but without diffusion.
Moreover, the reaction rate depends only on the 
global state of the system, but not on specific locations. 
Compared to these works, the stochastic systems that we study allow for two particles to change type simultaneously with 
location-dependent reaction rate,
in a reaction of the form $S_1 + S_2 \to S_3 + S_4$, for example. 
However, the total number of particles is conserved in our systems.

The derivation of macroscopic equations from  interacting particle systems is a classical topic in mathematical physics dating back to Maxwell and Boltzmann. A prototype microscopic model that has been studied a lot in the literature is the following McKean-Vlasov system of stochastic differential equations:
\begin{equation}\label{MV-eq}
dX_t^i= \sigma(t, X^i_t) dB_t^i+\f 1 N \sum_{j=1,j\neq i}^N F(X_t^i-X_t^j)d t,\quad 1\le i\le N.
\end{equation}
For a Lipschitz continuous force $F$ and a diffusion coefficient $\sigma$, this system has a {\em mean field limit} which is a nonlocal drift-diffusion equation (usually called the McKean-Vlasov equation); see e.g. \cite{Sznitman91}
and \cite{Glose16} for a proof. An important concept in studying the particle system \eqref{MV-eq}, and others like it, is {\em propagation of chaos}, due to  Kac \cite{Kac56} and McKean \cite{Mckean66, Mckean67} (see also the classical monograph by Sznitman \cite{Sznitman91}), which roughly means that as $N \to \infty$, any finite group of the particles are asymptotically chaotic (independent) if they are initially chaotic (independent); see Definition \ref{defn:kac} for a precise statement.  In recent years, some progress has been made in proving mean field limits
of McKean-Vlasov systems with singular interaction forces or kernel; we refer the interested reader to
\cite{FHM14,JW16,HJ15,JW18,serfaty2018mean}.
Among these results, we specifically mention the recent work by Jabin and Wang \cite{JW18}, who investigate the mean field limit
for stochastic systems \eqref{MV-eq} with $F\in W^{-1,\infty}$ which in particular includes the Biot-Savart kernel.
They obtain a quantitative estimate for the propagation of chaos in terms of the relative entropy between the $k$-marginals of
the joint distribution of $N$ particles and the $k$-tensorized distribution of the mean field limit.   
The relative entropy method was initiated by
Yau \cite{Yau91} in the proof of hydrodynamic limits of Ginzburg-Landau Models, where the relative entropy of the current density with respect to some local Gibbs equilibrium serves as a controllable norm whose convergence to zero implies the existence of macroscopic limit; see \cite{KL99,V98} for more discussion about this method in interacting particle models from statistical physics.
The recent work \cite{JW18}  demonstrated the potential power of the relative entropy method
in the study of mean field limits of singular stochastic dynamics. Compared to \cite{Yau91}, the relative entropy approach used in \cite{JW18} (and in the present article) is based on the idea of  controlling directly the relative entropy between the density of the particle system and that of the mean field limit, instead of the local equilibrium measure of the particle system.

Our strategy of proving propagation of chaos in the chemical reaction model is similar to that of Jabin and Wang \cite{JW18}. Specifically,
we show the propagation of chaos by proving an explicit estimate for the relative entropy between the
joint law of particle system and the tensorized law of the mean field limit.
The key ingredient of the proof is
a large deviation inequality of the form
\begin{align}\label{ldp-est}
\sup_{N\ge 2} \mbe\sqb{\exp\rb{\f 1 N \sum_{i,j=1,i\neq j}^N f(Y_i,Y_j) }  } <\infty,
\end{align}
where $\{Y_i\}_{i\in \mbn}$ are i.i.d. random variables, and $f$ is a $L^\infty$ function satisfying some appropriate cancellation condition (see Lemma \ref{Concen}).
Note that this inequality requires weaker regularity condition than the usual condition
in obtaining large deviation bounds through Varadhan's lemma (c.f. \cite[Theorem 2.5]{V16} and \cite[Theorem 1.2.1]{DE97})
since the function $f$ needs not be continuous.
The inequality \eqref{ldp-est} (with a weaker assumption on $f$) was proved by  Jabin and Wang (see \cite[Theorem 4]{JW18}) using combinatorial estimates. One of the contributions of this work is to provide a short probabilistic proof, assuming $f$ is bounded. Our proof relies on identifying a martingale structure in the exponential of \eqref{ldp-est} which allows us to conclude the estimate
from a sharp Marcinkiewicz-Zygmund inequality for a martingale difference sequence.

One may wonder why we only restrict our attention to the case of bimolecular reactions \eqref{bi-reaction} (two inputs, two outputs), 
instead of a more general reaction structure. 
First, although one may still discuss the mean field limit for general reaction structure, 
the notion of the propagation of chaos only applies to those systems that conserve the total number of particles in time.
Since our work is based on this notion, it excludes the case of those reactions with uneven inputs and outputs. 
It is also worthwhile to consider chemical reactions taking $m$ inputs, $m$ outputs for $m\ge 3$. However, proving a mean field limit 
for these models requires a delicate estimate of the  quantity \eqref{ldp-est}, which now involves $m$-body interactions.
This turns out to be highly-nontrivial and will be investigated in a future work.

The rest of the paper is organized as follows.
In the coming section, we begin with a brief introduction to general notions of chaos, chemical reaction networks and some notations, then we set up the $N$-particles
system and state the main result of its propagation of chaos property (Theorem \ref{main}).
In Section \ref{generator}, we will determine the infinitesimal generator $\bL_N$ of the process $(\bX^N_t,\bXi^N_t)=\{(X_t^i,\Xi_t^i)\}_{1\le i\le N}$ 
defined by the $N$-particles system, and its dual $\bL_N^*$.
The proof of the main theorem is presented in Section \ref{Int-ineq}, which is based on the idea of establishing appropriate differential
inequality for the normalized relative entropy and then applying Gr\"onwall lemma.
The large deviation inequality and a lemma used in the proof of this main theorem will be proved in Section \ref{LDP-ineq}.
An appendix, briefly discusses the well-posedness and regularity of the mean field limit equation and the Fokker-Planck equation for $N$-particles system, is also included in Section \ref{appen}.

\subsection*{Acknowledgement}

The work of James Nolen was partially funded through grant DMS-1351653 from the National Science Foundation.

\section{Settings and the Main Result}\label{setting}

\subsection{Bimolecular chemical reaction network}

To describe the model, we adopt the notations from the theory of chemical reaction networks (see e.g. \cite{AK15}), so let us now state the definition of general chemical reaction networks.
A \emph{chemical reaction network} is a triplet $(\fs,\fc,\fr)$ given as follows.
\begin{enumerate}
	\item $\fs=\{1,2,\cdots,n\}=\{S_1,\cdots,S_n\}$ is a finite set of chemical species.
	\item $\fc=\{C_1,C_2,\cdots, C_{n_c}\}\subset \mbn_0^n$ is a finite collection of  chemical complexes. For instance, 
	$C=\sum_{k=1}^n \al_k\fe_k\in \mbn_0^n$ (where $\{\fe_k\}_{k=1,\cdots, n}$ is the unit vector in $\mbr^n$, $\al_k\in \mbn_0$) 
	represents the complex formed by a total of $\al_k$ type-$S_k$ particles for each $1\le k\le n$.
	\item $\fr=\{R_1,R_2,\cdots, R_{n_r}\}\subset \fc\times \fc$ is a finite collection of chemical reactions. 
	For instance, $R=(C,C')\in\fr$ denotes the chemical reaction between complexes $C\to C'$. 
	If $C=\sum_{k=1}^n \al_k\fe_k, C'=\sum_{k=1}^n \be_k\fe_k$, then the reaction can be written in the form of a chemical equation
	\begin{align*}
	\sum_{k=1}^n \al_k S_k\xrightarrow{R} \sum_{k=1}^n \be_k S_k.
	\end{align*}
	The complex $C$ (the left hand side of the above) will be called \emph{the input of the reaction $R$}, whereas $C'$ (the right hand side) is \emph{its output}.
\end{enumerate}

In this work, we restrict our attention to the case of a \emph{bimolecular reaction network}, which is a particular case of a chemical reaction network,
whose complexes $C\in \fc$ are of bimolecular form, namely,
\begin{align*}
\fc\subset \{ \fe_{k}+\fe_{l}: k,l\in \fs \}.
\end{align*}
This means that all reactions $R\in \fr$ have the form $S_k+S_l \xrightarrow{R}  S_{k'}+S_{l'},$ for some $k,l,k',l'\in \fs$. In this case, we use the notation
\begin{align}\label{chem-reac}
k+l\xrightarrow{R} k'+l',\qquad (k,l)\xrightarrow{R}(k',l'),\qquad R:(k,l)\to (k',l').
\end{align}
It is possible that $k = l$ and/or $k' = l'$, so that the two input molecules and/or the two output molecules are of the same species (e.g. \eqref{special}).
Throughout this work, for every chemical reaction equation $R\in \fr$ written as above, to avoid ambiguity we always assume \emph{the ascending order for both input and output}, that is,
\begin{align}\label{convention}
k\le l,\qquad k'\le l'.
\end{align}
For $R\in \fr$, we denote $R^-$ (resp. $R^+$) the input (resp. output) of the reaction. For instance, a chemical reaction of the form \eqref{chem-reac}, we have
\begin{align}\label{in-output-red}
R^-=\{k,l\},\qquad R^+=\{k',l'\}.
\end{align}
We remark that the number of particles, and hence the total mass, are conserved in every bimolecular reaction.

\subsection{Notations}
We now introduce some general notations being used in the work.
Throughout this paper, we will consider stochastic processes in the state space
\begin{align*}
\Pi =\mbx\times \fs,
\end{align*}
where $\mbx=\mathbb{T}^d$ (the $d$-dimensional torus), and the discrete set $\fs =\{1,2,\cdots, n\}$ represents the type of particles.
Though the present work is mainly based on the case of a torus, we sometimes will also consider the whole space setting, i.e., $\mbx =\mbr^d$. 
We set the variable $y=(x,\xi)\in \Pi$, with $x\in \mbx$, $\xi\in \fs$, and denote
\begin{align*}
dm=dx\otimes d\#
\end{align*}
the canonical measure on $\Pi$, that is, the product of Lebesgue measure on $\mathbb{T}^d$ (or $\mbr^d$), and the counting measure $\#$ on $\{1,2,\cdots,n\}$.
For $N\ge 1$, denote variables
\begin{align*}
\by_N&=(y_1,\cdots, y_N)=(\bx_N,\bxi_N)\in \Pi^N,\\
y_k&= (x_k,\xi_k)\in \Pi=\mbx\times \fs,\qquad  k\in\{1,2,\cdots, N\},\\
\bx_N&=(x_1,\cdots,x_N)\in \mbx^N,\\
\bxi_N&=(\xi_1,\cdots, \xi_N)\in\fs^N.
\end{align*}
Specifically, the boldface symbols $\by_N,\bx_N,\bxi_N$ are for elements in the $N$-fold product spaces $\Pi^N,\mbx^N,\fs^N$, 
while the normal typeset symbols $y,x,\xi$ denote elements in $\Pi,\mbx,\fs$.
Likewise, the boldface uppercase symbols (e.g., $\bX_t^N,\bXi_t^N$) denote processes on the $N$-fold product spaces, 
while the normal uppercase symbols (e.g., $X_t,\Xi_t$) denote processes on $\mbx,\fs$.
Also, let $m^N$ be the $N$-fold product of the measure $m$, which is a measure on $\Pi^N$:
\begin{align*}
m^N:=m^{\otimes N}=\overbrace{m\otimes \cdots \otimes m}^N.
\end{align*}

Finally, for $N\ge 1,p\in[1,\infty]$, and $k\ge 0$, denote the spaces of measures and functions on $\Pi^N$ as below:
\begin{itemize} \setlength\itemsep{0.45em}
	\item $\mathcal M(\Pi^N)$ - the family of Borel measures on $\Pi^N$;
	\item $\mathcal P(\Pi^N)$ - the family of probability measures on $\Pi^N$;
	\item $C_0(\Pi^N)$ - the space of continuous functions on $\Pi^N$, vanishing at infinity;
	\item $L^p(\Pi^N)$ - the space of all $L^p$- integrable functions on $\Pi^N$ (w.r.t. $dm^N$);
	\item $C_0^k(\Pi^N)$ - the space of functions $f$ so that $f(\cdot,\bxi_N)\in C_0^k(\mbx^N)$ for all $\bxi_N\in \fs^N$;
	\item $W^{k,p}(\Pi^N)$ - the space of functions $f$ so that $f(\cdot,\bxi_N)\in W^{k,p}(\mbx^N)$ for all $\bxi_N\in \fs^N$. 
\end{itemize}
Also, define a bilinear form $\inn{\cdot,\cdot}:C_0(\Pi^N)\times \mathcal{M}(\Pi^N)\to \mbr$ by
\begin{align}\label{bilinear}
\inn{\varphi,\rho}&:= \int_{\Pi^N} \varphi(\by_N)d\rho(\by_N).
\end{align}

\subsection{The stochastic chemical reaction-diffusion system}
Fix dimension $d\ge 1$ and a bimolecular chemical reaction network $(\fs,\fc,\fr)$, with $\fs=\{1,2,\cdots, n\}$.
Fix a function
\begin{align}\label{si-def}
\si:\fs \to \mbr^+,
\end{align}
to represent the diffusion coefficient for each type of particle.
For each reaction $R\in \fr$ of the form \eqref{chem-reac}, we associate a non-negative function
\begin{align}\label{Phi-def}
\Phi_R\in L^1(\mbx)\cap L^\infty(\mbx),\quad \mbox{with }\Phi_R(x)=\Phi_R(-x) \mbox{ for all } x\in \mbx,
\end{align}
which will be called \emph{the reaction kernel of the chemical reaction $R$}.
The reaction kernel $\Phi_R$ will be used to define the rate at which the reaction $R$ between two particles occurs. One typical choice of reaction kernel, as introduced in \cite{Doi76a}, could be a cut-off function
supported on finite ball, i.e $\Phi_R(x) = \chi_{\{|x|\leq r\}}$ for some $r>0$. The symmetry assumption \eqref{Phi-def} is not essential to us and it is 
only used for the purpose of simplifying expressions in the mean field limit; see Section \ref{sec:formalmfl} for more details.

Next, we introduce some useful indicator functions. Given a fixed type $k\in \fs$, we denote by $\chi_k:\fs \to \{0,1\}$ the indicator
\begin{align*}
\chi_k(\xi)=\left\{ \begin{array}{ll}
1& \mbox{if }\xi=k,\\
0& \mbox{else}.
\end{array} \right.
\end{align*}
For $R\in \fr$ with $(k,l)\xrightarrow{R}(k',l')$, let $\chi_R^\pm:\fs^2\to\{0,1\}$ be the indicator function
\begin{align}\label{chi-def}
\chi_R^-(\xi,\xi')= \chi_{k}(\xi)\chi_l(\xi'),\qquad \chi_R^+(\xi,\xi')= \chi_{k'}(\xi)\chi_{l'}(\xi').
\end{align}
That is, $\chi^-_R(\xi,\xi')$ (resp. $\chi^+_R$) indicates the event when $(\xi,\xi')$ matches with the input $R^-$ (resp. output $R^+$) of the reaction $R$.
We stress here these indicators are \emph{order sensitive} when $k \neq \ell$ and $k' \neq \ell'$, in the sense that
\begin{align}\label{order}
\chi_R^-(k,l)=1,\quad \chi_R^-(l,k)=0,\quad \chi_R^+(k',l')=1,\quad \chi_R^+(l',k')=0.
\end{align}

We next introduce the ``random ingredients" for the model.
Let $i\in \mbn$ label particles involved in the modeled chemical reaction-diffusion process.
For each particle $i$, we associate the following independent random variables and stochastic processes to it:
\begin{enumerate}
	\item a $\Pi$-valued random variable $(X_0^i,\Xi^i_0)$;
	\item a standard Brownian motion $\{B^i_t\}_{t\ge 0}$.
\end{enumerate}
For each $R\in \fr$ and \emph{ordered
 }pair $(i,j)\in \mbn^2$ with $i\neq j$, we associate also
\begin{enumerate}
	\item[(3)] an independent unit rate Poisson process $\{E^{ij}_R(t)\}_{t\ge 0}$;
\end{enumerate}
These Poisson processes will be used to define the counts, up to a certain time, of the type-$R$ reactions that happens between the (ordered) pair of distinct particles $i$ and $j$.
The collection of processes $\{B_t^i\}_{t\ge 0}, \{E^{ij}_R(t)\}_{t\ge 0}$, for $i,j\in \mbn$ and $R\in \fr$ are assumed to be independent.
Let $(\Om,\mclf,\mathbb{P})$ be the probability space on which these random variables and processes are defined.

Fix $N\gg1$ large and consider the stochastic system of $N$-particles described as follows.
For each $i\in\{1,\cdots, N\}$, we consider the process $\{(X_t^i,\Xi^i_t)\}_{t\ge 0}$ on the state space $\Pi=\mbx\times \fs$, with $X^i_t\in \mathbb{X}$ representing the location of the $i$-th particle at time $t\ge 0$, and $\Xi^i_t\in \fs=\{1,2,\cdots,n\}$ representing its type.
Each particle $i$ diffuses in the spatial domain $\mbx$ independently, with diffusion coefficient $\si(\Xi^i_t)$, depending on its current type $\Xi^i_t$. Therefore, $X_t^i$ satisfies the SDE
\begin{align*}
X_t^i=X_0^i+\int_0^t \si(\Xi_s^i)dB_s^i,\qquad 1\le i\le N.
\end{align*}
Next consider the type process $\{\Xi_t^i\}$ for a given particle $i$, which is a pure jump process.
The type has initial value $\Xi_0^i$.
To describe the evolution, we introduce the \emph{reaction counter process} $\til E^{ij}_R(t)$ for a given reaction $R\in \fr$ and the (ordered) pair of particles $(i,j)$ (with $i\neq j$), which is the following time-change of the Poisson process $E^{ij}_R(t)$:
\begin{align*}
\til E^{ij}_R(t)&= E^{ij}_R \rb{ \f 1 {N} \int_0^t \chi_R^-(\Xi_{s-}^i,\Xi_{s-}^j) \Phi_R(X_s^i-X_s^j) ds } .
\end{align*}
Specifically, this is a counting process that counts the number of type-$R$ reactions occurring between the pair $(i,j)$ up to time $t$.
The rate of these reactions depends on the relative locations $X_t^i,X_t^j$, through the reaction kernel $\Phi_R$, and is given by
\begin{align*}
\f 1 N \cdot \chi_R^- (\Xi_{t-}^i,\Xi_{t-}^j ) \cdot \Phi_R(X_t^i-X_t^j).
\end{align*}
The scale $\f 1 N$ is the virtue of the mean field interaction, so that the pairwise interaction diminishes as the number of particles $N\to\infty$.
Observe that the rate is nonzero only when $\chi_R^-(\Xi_{t-}^i,\Xi_{t-}^j)=1$ and when $\Phi_R(X_t^i-X_t^j)\neq 0$. In other words,
a reaction of type-$R$ can happen between the pair $(i,j)$ only when the type $(\Xi_{t-}^i,\Xi_{t-}^j)$ matches with the input $R^-$ of
the reaction (i.e., $\Xi_{t-}^i=k,\Xi_{t-}^j=l$, see  \eqref{chi-def}), and when their locations $X_t^i,X_t^j$ are sufficiently close (if $\Phi_R$ has a compact support).

As time progresses, a jump $\Xi_{t-}^i\to \Xi_t^i$ occurs when a reaction $(k,l)\xrightarrow{R} (k',l')$ between particle $i$ and 
some other particle $j\neq i$ takes place. Namely, a jump in $\Xi_t^i$ happens at the time $t$ when either a reaction of type-$R$ between the pair $(i,j)$,
or $(j,i)$ occurs, and hence, when either reaction counter $\til E_R^{ij}(t)$ or $\til E^{ji}_R(t)$ jumps.
When that happens, the type $(\Xi_{t-}^i,\Xi_{t-}^j)$ instantly turns into the output $R^+$ of the reaction $R$: either $(k',l')$ or $(l',k')$.
To avoid ambiguity, we enforce the rule of assignment that \emph{if the reaction of type-$R$ occurs between the pair $(i,j)$ of particles,
the type $(\Xi_{t-}^i,\Xi_{t-}^j)$ turns into $(k',l')$ in order} (recall the convention \eqref{convention} that $k'\le l'$).
This also means, if a reaction occurs between the pair $(j,i)$ instead, and hence $(\Xi_{t-}^j,\Xi_{t-}^i)=(k,l)$, it also turns into the same output: $(\Xi_{t}^j,\Xi_{t}^i)=(k',l')$.
Specifically, if $(\Xi_{t-}^i,\Xi_{t-}^j)$ has value $(k,l)$ (resp. $(l,k)$), then the reaction $\til E^{ji}_R$ (resp. $\til E^{ij}_R$) will not fire,
as the rate $\chi_R^-(\Xi_{t-}^j,\Xi_{t-}^i)=0$ (resp., $\chi_R^-(\Xi_{t-}^i,\Xi_{t-}^j)=0$ see \eqref{order}).
Therefore, with the rule of assignment, when a reaction of type-$R$ happens, it always turns types $k\to k'$, $l\to l'$ in ascending order.
In total, the evolution of the type process is described by the following stochastic integrals against the reaction counter processes $\til E^{ij}_R$
\begin{align*}
\Xi^i_t&= \Xi^i_0+ \sum_{\substack{R\in \fr \\ R:(k,l)\to(k',l')}} \sum_{j=1,j\neq i} ^N \sqb{\int_0^t (k'-\Xi_{s-}^i) d \til E^{ij}_R(s)+ \int _0^t (l'-\Xi_{s-}^i)d\til E^{ji}_R(s)}.
\end{align*}

In summary, the dynamics of the process $(\bX^N_t,\bXi^N_t)=\{(X^i_t,\Xi_t^i)\}_{1\le i\le N}$ is described by the following SDE system:
\begin{align}\label{SDE}
\left\{ \begin{array}{rl}
X_t^i&= X_0^i + \displaystyle \int _0^t \si(\Xi^i_s)dB^i_s, \\
\Xi^i_t&=\Xi_0^i +\displaystyle \sum_{\substack{R\in \fr \\ R:(k,l)\to(k',l')} } \sum_{j=1,j\neq i} ^N  \sqb{\int_0^t (k'-\Xi_{s-}^i) d \til E^{ij}_R(s)+ \int _0^t (l'-\Xi_{s-}^i)d\til E^{ji}_R(s)},\\
&\qquad \qquad \mbox{for }i\in\{1,2,\cdots,N\},\\
\til E^{ij}_R(t)&= E^{ij}_R \rb{ \displaystyle  \f 1 {N} \int_0^t \chi^-_R(\Xi_{s-}^i,\Xi_{s-}^j) \Phi_R(X_s^i-X_{s}^j) ds } .
\end{array} \right.
\end{align}
Observe that $\Xi_t^i$ takes values in $\{1,\cdots, n\}$.  The process $t \mapsto \int_0^t (k'-\Xi_{s-}^i) d \til E^{ij}_R(s)$ is 
piecewise constant, making a jump of size $(k' - k)$ only at times when $\Xi^i_t = k$.  With probability one, the jumps in $\tilde E^{ij}_R(t)$ occur at distinct times.

We can also write this equation in vector form. Let $\boldsymbol{B}_t^N=(B_t^1,\cdots, B_t^N)$, 
which is the standard Brownian motion on $\mbx^{N}$, and $\Si_N(\bXi_t^N)$ be the $(N\times N)$-diagonal matrix with diagonal entries $\si(\Xi_t^i)$, for $1\le i\le N$.
Introduce also the function $\Ta ^{ij}_R:\fs^N\to \fs^N$ by
\begin{align}\label{Ta-def}
\Ta_R^{ij}(\bxi_N)=\Ta_R^{ij}(\xi_1,\cdots, \xi_N)=(\til \xi_1,\cdots, \til \xi_N),\qquad
\til \xi_m&=\left\{ \begin{array}{ll}
k'& \mbox{ if }m=i,\\
l'& \mbox{ if }m=j,\\
\xi_l& \mbox{else}.
\end{array} \right.
\end{align}
That is, $\Ta_R^{ij}(\bxi_N)$ is obtained by changing only the $i$- and $j$-coordinate of $\bxi_N$ from $\xi_i,\xi_j$ to 
the output $k',l'$ of the reaction $R$ respectively, while leaving the other coordinates unchanged.
For later use, we also introduce $\til \Ta_R^{ij}:\fs^N\to \fs^N$ the ``reverse" of the map $\Ta_R^{ij}$, which turns the $i,j$-coordinates into the input $k,l$ of $R$:
\begin{align}\label{Tta-def}
\til \Ta_R^{ij}(\bxi_N)=\til\Ta_R^{ij}(\xi_1,\cdots, \xi_N)=(\til \xi_1,\cdots, \til \xi_N),\quad \til \xi_m=\left\{ \begin{array}{ll}
k& \mbox{if }m=i,\\
l& \mbox{if }m=j,\\
\xi_m& \mbox{otherwise}.
\end{array} \right.
\end{align}
Then the vector form of \eqref{SDE} is written as follows:
\begin{align}\label{SDE-V}
\left\{ \begin{array}{ll}
\bX_t^N&= \bX_0^N+ \displaystyle \int_0^t \Si_N(\bXi_s^N)d\boldsymbol{B}_s^N,\\
\bXi_t^N&= \bXi_0^N+\displaystyle \sum_{R\in \fr} \sum_{i,j=1,i\neq j}^N \int_0^t [\Ta^{ij}_R(\bXi_{s-}^N)-\bXi_{s-}^N]dE_R^{ij}(s).
\end{array} \right.
\end{align}

\begin{rem}
	Before ending this subsection, let us make one comment about the rule of assignment. 
	By our construction, whenever a reaction $R:(k,l)\to(k',l')$ takes place, it always turns
	$k\to k', l\to l'$, following the ascending order of the input and output.
	Perhaps a more realistic situation will be turning the input into a prescribed choice of output. 
	For instance, when considering a reaction $R:(1,2)\to(2,3)$, it is more natural to have $1\to3$, while 2 remains $2\to2$, 
	instead of following the ascending order $1\to 2,2\to3$. To be precise, for those reactions $R:(k,l)\to(k',l')$ with distinct inputs $k\neq l$, 
	we prescribe a surjective assignment $\ta_R:\{k,l\}\to\{k',l'\}$.
	Whenever a reaction of type-$R$ occurs, it turns $k\to \ta_R(k)$, and $l\to \ta_R(l)$. If $k=l$, 
	we then switch the types into $(k',l')$ or $(l',k')$, each with probability $\f 12$.
	This is equivalent to first fix an order of the right hand side (the output) of \eqref{chem-reac} for each $R$, namely,
	\begin{align*}
	k+l\xrightarrow{R}\ta_R(k)+\ta_R(l),
	\end{align*}
	(the order of the output does not matter when $k=l$).
	Then a reaction turns the input $(k,l)$ into the output, following the order of the right hand side above.
	One may construct the dynamics with this rule, and of course the result of propagation of chaos still holds,
	except with a minor difference in the mean field limit equation \eqref{MFE} (more precisely, the definition of $\bar T_R^+$ in \eqref{TR-def}).
\end{rem}

\subsection{Notions of entropy and chaos.}
Before stating our main results, we define some notions of entropy and chaos.
For a generic Polish space $\Pi$, let $\mathcal{P}(\Pi)$ be the family of probability measures on $\Pi$. For $N \geq 2$, we denote by
$\mathcal{P}_{sym}(\Pi^N)$ the set of symmetric probability measures on the product space $\Pi^N $, that is the set of laws
of exchangeable $\Pi^N$-valued random variables. Given a symmetric probability density $\rho_N \in \mathcal{P}_{sym} (\Pi^N)$,
let us denote by $\rho_{k,N} \in \mathcal{P}(\Pi^k)$, $1\leq k \leq N$, the $k$-marginal of $\rho_N$.
For two probability measures $\mu, \nu \in \mathcal{P}(\Pi)$, the \emph{relative entropy of $\mu$ with respect to $\nu$} is defined as
$$
\mathcal{H}(\mu \| \nu) = \mathbb{E}^{\mu} \log \Big(\frac{d\mu}{d\nu}\Big).
$$
If $\mu$ and $\nu$ have densities $f,g$ with respect to a measure $m$ on $\Pi$, then
\begin{align}\label{RE1}
\mch(\mu\|\nu)&= \int_{\Pi} f(y)\log\rb{\f {f(y)}{g(y)} }  dm(y).
\end{align}
For two symmetric probability measures $\rho_N, \bar{\rho}_N \in \mathcal{P}_{sym} (\Pi^N)$ with $k$-marginals
$\rho_{k,N} $ and $ \bar{\rho}_{k,N}$, we also define the following \emph{normalized relative entropies}
$$
\mch_N(\rho_N \| \bar{\rho}_N)  := \frac{1}{N}\mch(\rho_N \| \bar{\rho}_N), \quad
\mch_k(\rho_N \| \bar{\rho}_N)  := \frac{1}{k}\mch(\rho_{k,N} \| \bar{\rho}_{k,N}).
$$

\begin{definition}[Kac's chaos]\label{defn:kac}
	Let $\rho_N \in \mathcal{P}_{sym}(\Pi^N)$ be a sequence of symmetric probability measures on $\Pi^N, N\geq 1$ and let $\bar \rho\in \mathcal{P}(\Pi)$.
	We say $\rho_N$ is \emph{$\bar \rho$-Kac's chaotic}
	if one of the following three equivalent conditions holds:

	(i) the sequence of two marginals $\rho_{2,N} \rightharpoonup \bar\rho\otimes \bar \rho$ (converges weakly) as $N \rightarrow \infty$;

	(ii) for all $k \geq 1$ fixed,
	\begin{align*}
	\rho_{k,N} \rightharpoonup \bar \rho^{\otimes k} \text{ on } \Pi^k \text{ as } N\rightarrow \infty;
	\end{align*}

	(iii) the law $\hat{\rho}_N$ of the empirical measure $\mu^N_{\bX^N}$ associated to $\bX^N = \{X^i\}_{i=1}^N \sim \rho_N$ satisfies
	$$
	\hat{\rho}_N \rightharpoonup \delta_{\bar \rho} \text{ in } \mathcal{P} (\mathcal{P}(E)) \text{ as } N\rightarrow \infty.
	$$
\end{definition}
The first notion of chaos given in Definition \ref{defn:kac} (ii) was defined by Kac \cite[Section 3]{Kac56}.
Sznitman \cite{Sznitman91} proved the equivalence of the three formulations above. We also refer to \cite[Theorem 1.2]{HM14}
for more quantitative
statements about Kac's chaos.

\commentout{
\begin{definition}[Propagation of Kac's chaos]
	Let $\rho_N(t)$ be the probability density of a $\Pi^N$-valued process $\bY_t$  (e.g., a solution of \eqref{SDE-V}), 
	with initial distribution $\rho_N(0)$.  Let $\bar \rho(t):[0,T] \to \mathcal{P}(\Pi)$.	Assume that the sequence $\rho_N(0)$ is $\bar{\rho}_0$-Kac's chaotic. Then we say
	that the dynamics satisfies the \emph{propagation of chaos property
	on $[0,T]$}
	if for every $t \in [0,T]$, $\rho_N(t)$ is  $\bar{\rho}(t)$-Kac's chaotic.
\end{definition}
}
The primary goal of this paper is to prove that the law of the chemical reaction
system \eqref{SDE-V} is Kac's chaotic for all $t \geq 0$. We will achieve this goal by proving a quantitative bound on the normalized relative entropy between the joint distribution of the $N$-particles and the tensorized law of the mean field limit; see Theorem \ref{main}.

\subsection{Main result}
The process $\{(\bX_t^N,\bXi_t^N)\}_{t\ge 0}$ defines a Markov process on the state space $\Pi^N$, whose generator $\bL_N$ and adjoint $\bL_N^*$ will be determined in the next section.
Thus, the joint probability distribution $\rho_N(t)$ (for $t\ge 0$) of the process $(\bX_t^N,\bXi^N_t)$, which is a probability measure on $\Pi^N$, satisfies the Fokker-Planck equation
\begin{align*}
\partial _t \rho_N = \bL_N^* \rho_N, \qquad  t\in (0,\infty).
\end{align*}
This equation is in fact a linear parabolic system.
From the theory of parabolic equations, $\rho_N$ admits a smooth (in $\bx_N$) density w.r.t. the measure $m^N$, provided that the initial data $\rho(0)$ admits
a nonnegative $L^1$-density; see Section \ref{appen}, specifically, Proposition \ref{well-posed}.
Abusing notation, we will use the same symbol $\rho_N(t,\cdot)$ to denote its density.

Our main objective is to show, as $N\to\infty$, the joint distribution $\rho_N$ converges, in some sense, to the tensorization $\bar \rho_N=\bar \rho^{\otimes N}$ of the solution of the mean field limit equation $\bar \rho$.
The mean field limit equation of this stochastic system, which will be formally derived in Section \ref{generator}, is the following \emph{nonlocal reaction-diffusion system}
\begin{align} \label{limitsystem}
\left\{ \begin{array}{ll}
\partial _t u_\xi &= \displaystyle  \f {\si(\xi)^2}{2}\De u_\xi + F_\xi^- +F_\xi^+, \\
F_\xi^-&= -\displaystyle \sum_{\substack{R\in \fr: R^-=\{k,l\}}} [\de_{\xi k}(\Phi_R*u_l)u_k+ \de_{\xi l}(\Phi_R*u_k)u_l],\\
F_\xi^+&= \displaystyle \sum_{\substack{R\in \fr\\ R:(k,l)\to(k',l') } }[\de_{\xi k'}  \Phi_R* u_{l})u_{k}+ \de_{\xi l'}(\Phi_R* u_{k})u_{l}], \\
u_\xi(0,x)&=u_{0,\xi}(x),\qquad \qquad 1\le \xi\le n,
\end{array} \right.
\end{align}
where $\bar \rho=(u_1,\cdots, u_n)$, namely, $\bar \rho(t,x,\xi )=u_\xi(t,x)$ for $1\le \xi \le n$, and $\de_{\xi l}$ denotes the Dirac delta, and the operator $\De$ is
the Laplacian in the spatial coordinate $x$.  The initial data $\bar \rho_0=(u_{0,1},u_{0,2},\cdots, u_{0,n})$ is a probability density (with respect to $m$) on $\Pi$, namely, 
\begin{align}\label{initial}
\bar \rho_0(x,\xi)=u_{0,\xi}(x)\ge 0\, \mbox{ for }1\le \xi\le n,\qquad \int_{\Pi}\bar \rho_0(y)dm(y)= \sum_{\xi=1}^n\int_{\mathbb{X} } u_{0,\xi}(x)dx=1.
\end{align}
Alternatively, $\bar \rho$ satisfies the equation
\begin{align}\label{MFE}
\left\{ \begin{array}{ll}
\displaystyle \partial _t \bar \rho = \f{\si(\xi)^2}{2}\De \bar \rho +\bar T(\bar \rho) ,& (t,x,\xi)\in(0,\infty)\times \mbx\times \fs,\\
\bar \rho(0,x,\xi)= \bar\rho_0(x,\xi),&
\end{array} \right.
\end{align}
where $\bar T$ is the (nonlinear) operator
\begin{align}\label{T-def}
\bar T&= \sum_{R\in \fr} \bar T_R=\sum_{R\in \fr}(\bar T_R^-+\bar T_R^+),
\end{align}
and $\bar T_R^\pm$, for a reaction $R\in \fr$ with $(k,l)\xrightarrow{R} (k',l')$, is given by
\begin{align}\label{TR-def}
\bar T_R^-(\bar \rho)(x,\xi)&= -\sqb{ \chi_{k}(\xi)(\Phi_R*\bar \rho)(x,l)\bar \rho(x,k) + \chi_l(\xi)(\Phi_R*\bar \rho)(x,k)\bar \rho(x,l) },\\
\bar T_R^+(\bar \rho)(x,\xi)&=\chi_{k'}(\xi)(\Phi_R*\bar \rho)(x,l) \bar \rho(x,k)+ \chi_{l'}(\xi)(\Phi_R*\bar \rho)(x,k)\bar \rho(x,l) . \nonumber
\end{align}
This nonlinear parabolic system \eqref{limitsystem} is a regularized version of \emph{the local reaction-diffusion system} (\eqref{limitsystem}
with the Dirac delta measure $\la _R\de_0$, $\la_R>0$, in place of $\Phi_R$).
The system is in fact globally well-posed for every nonnegative initial data $\bar \rho_0\in L^1(\Pi)$ from \eqref{initial}. Moreover, these solutions are nonnegative and regular, with the total mass $\int_{\Pi} \bar \rho(t,y)dm(y)$ conserved in time. Specifically, solutions are regular in the sense that
\begin{align*}
\bar \rho\in C([0,\infty);L^1(\Pi ))\cap C((0,\infty);W^{2,p}(\Pi))\cap C^1((0,\infty);L^p(\Pi)),\quad  \forall p\in [1,\infty).
\end{align*}
Thus, the solution $\{\bar \rho(t,\cdot)\}_{t\ge 0}$ is a time-dependent probability density on $\Pi$.
The proof of these results (well-posedness, regularity) will be presented in the appendix (Section \ref{appen}).

Let $\bar \rho_N= \bar \rho^{\otimes N}$ be the tensorized law of $\bar \rho$, namely,
\begin{align*}
\bar \rho_N(t,\by_N)&= \prod_{i=1}^N \bar \rho(t,y_i),\quad \by_N=(y_1,\cdots, y_N)\in \Pi^N.
\end{align*}
Again, our goal is to show that the distribution $\rho_N$ of $(\bX_t^N,\bXi^N_t)$ converges to $\bar \rho_N$ in some appropriate sense.
To this end, we consider the renormalized relative entropy between $\rho_N,\bar\rho_N$:
\begin{align}\label{RE}
W_N(t)&= \mch_N(\rho_N\|\bar \rho_N)(t):= \f 1 N \int_{\Pi^N} \rho_N(t,\by_N)\log\rb{\frac{\rho_N(t,\by_N)}{\bar\rho_N(t,\by_N)} }  dm^N(\by_N),
\end{align}
(cf. \eqref{RE1}).
The main result of this work, stated as follows, 
shows that this quantity vanishes  uniformly on $[0,T]$ for any $T>0$ as $N\to\infty$, provided $W_N(0)\to0$.

\begin{theorem}\label{main}
	Let $d\ge 1$, $(\fs,\fc,\fr)$ be a bimolecular chemical reaction network, and $\mbx=\mathbb{T}^d$.
	Let $\si$, $\{\Phi_R\}_{R\in \fr}$ be the diffusion coefficients and reaction kernels from \eqref{si-def}, \eqref{Phi-def}.
	For $N\ge 2$, let $\rho_N$ be the law of the process $\{(\bX_t^N,\bXi_t^N)\}$ described by \eqref{SDE} with $\rho_N(0) \in \mathcal{P}_{sym}(\Pi^N)$. Let $\bar \rho_N=\bar \rho^{\otimes N}$,
	where $\bar \rho$ is the unique global solution to the mean field limit equation \eqref{MFE} with initial data $\bar \rho_0$ from \eqref{initial} 
	(guaranteed by Proposition \ref{well-posed}). 
	
	If $\bar \rho_0\in L^\infty(\Pi)$, $\inf_{(x,\xi)\in \Pi} \bar \rho_0(x, \xi)>0$, then for every $T>0$, there exists a 
	constant $\Lambda_T>0$ depending on  $\bar \rho_0$ such that the following estimate holds with $\|\Phi\|_{L^\infty}=\sum_{R\in \fr}\|\Phi_R\|_{L^\infty}$:
	\begin{align}
	\mch_N(\rho_N\|\bar \rho_N)(t)\le \frac{\left(e^{t \Lambda_T\|\Phi\|_{L^\infty}}  - 1\right)}{N}  +  e^{t \Lambda_T\|\Phi\|_{L^\infty}}  \mch_N(\rho_N\|\bar\rho_N)(0), \quad \forall \;\; t \in [0,T]. \label{RE-bound}
	\end{align}
\end{theorem}

As a direct consequence of the main theorem, we have the propagation of chaos property for the $N$-particle system \eqref{SDE}.
\begin{corollary}
	[Propagation of chaos]
	Assume the same settings as Theorem \ref{main}.
	If $\mch_N(\rho_N\|\bar\rho_N)(0) \rightarrow 0$ \text{ as } $N \rightarrow \infty$, then for each $t \in [0,T]$, $\rho_N(t)$ is $\bar{\rho}(t)$-Kac's chaotic (c.f. Definition \ref{defn:kac}). 
\end{corollary}

\begin{proof}
	First, by Lemma \ref{lem:exch}, the probability distribution $\rho_N(t)$ is symmetric for any $t\in [0,T]$. Therefore, thanks to the monotonicity \cite{wang2017mean} of the normalized relative entropy:
	$$
	\mathcal{H}_k(\rho_{k, N} \| \bar\rho^{\otimes k} ) \leq \mathcal{H}_N (\rho_N \|  \bar \rho^{\otimes N}), \quad \quad 1\leq k \leq N
	$$
	and the \emph{Csisz\'ar-Kullback-Pinsker inequality} \cite[Remark 22.12]{Villani09}:
	\begin{align}\label{CKP-ineq}
	\|f_1 - f_2\|_{L^1(\Pi)} \leq \sqrt{2\mch(f_1 \| f_2)}, \quad \forall f_i\in \mathcal{P}(\Pi),
	\end{align}
	one obtains by \eqref{RE-bound} and assumption that for any $k \geq 1$,
	\begin{align*}
	\sup_{t\in [0,T]} 	\|\rho_{k,N}(t) - \bar \rho^{\otimes k}(t)\|_{L^1(\Pi^k)} & \leq \sqrt{2k\cdot \mathcal{H}_k (\rho_{k,N} \|\bar \rho^{\otimes N})(t)}
	 \leq \sqrt{2k \cdot \mch_N (\rho_N | \bar \rho^{\otimes N})(t)}\\
	& \leq C(T,k) \sqrt{ \frac{1}{N} + \mch_N(\rho_N\|\bar\rho_N)(0)} \rightarrow 0   \text{ as } N\rightarrow \infty. \qedhere
\end{align*}
\end{proof}

\begin{rem}
1. With additional assumptions imposed for the mean field limit $\bar \rho$, one may in fact extend the main result to 
the whole space setting $\mbx=\mbr^d$; for example, the following: it holds for some $T,C_T>0$ that:
\begin{align}\label{comparability}
	\max_{R:(k,l)\to(k',l')} \sup_{(t,x)\in [0,T]\times \mbx}\cb{\frac{\bar \rho(t,x,k)}{\bar \rho(t,x,k')},\frac{\bar \rho(t,x,l)}{\bar \rho(t,x,l')} }\le C_T.
\end{align}
Although unnatural, this comparability assumption is critical in establishing the bound \eqref{RE-bound} (as $\Lambda_T$ depends on $C_T$). 
Indeed, in the case of torus $\mbx=\mathbb{T}^d$, the condition \eqref{comparability} alone, which is weaker than the assumption $\inf_{(x,\xi)\in \Pi}\bar \rho_0(x,\xi)>0$ 
imposed by Theorem \ref{main} (as a consequence of the maximum principle), is sufficient to imply the main result. 
Of course, to extend the result to $\mbr^d$, one may need more assumptions for $\bar \rho$, for instance,
some appropriate regularity and decay at infinity condition for $\bar \rho$, so that Lemma \ref{lem-REineq} holds. 
We are by no means to list down these precise assumptions, and we leave it to the interested reader.

\medskip
2. Since the relative entropy bound \eqref{RE-bound} from Theorem \ref{main} has an explicit dependence in the $L^\infty$-norm of reaction kernels, 
one may use that to obtain the convergence to mean field limits for local reaction-diffusion equations.
That is, we consider the stochastic system \eqref{SDE} with the reaction kernels $\Phi_R$ scaled according the the number of particles $N$.
More precisely, given $r>0$, let $\Phi_R^r\in L^1(\mbx)$ be the $L^1$-rescaled of the reaction kernel $\Phi_R$, namely,
\begin{align*}
\Phi_R^r(x)=r^{-d} \Phi_R\rb{r^{-1}x}.
\end{align*}
Let $\{(\bX^{N,r}_t,\bXi^{N,r}_t)\}_{t\ge0}$ satisfy the system \eqref{SDE-V} with $\Phi_R^r$ in place of $\Phi_R$, $\rho_N^r(t)$ 
be its distribution, and $\bar \rho^r$ be the solution of \eqref{MFE}, with $\Phi_R^r$ taking the role of $\Phi_R$ in \eqref{TR-def}.
Then the bound \eqref{RE-bound} implies
\begin{align*}
\sup_{t\in [0,T]}\mch_N(\rho_N^r\|(\bar\rho^r)^{\otimes N})(t)\le \frac{1}{N} e^{r^{-d} T \Lambda_T \|\Phi\|_{L^\infty}} + e^{r^{-d}T \Lambda_T \|\Phi\|_{L^\infty}} \mch_N(\rho_N^r\|(\bar\rho^r)^{\otimes N})(0),
\end{align*}
provided that the comparability condition \eqref{comparability} holds with $\bar \rho^r$ in place of $\bar \rho$, for some constant $C_T$ independent of $r>0$, namely,
\begin{align}\label{r-comp}
\sup_{r\in[0,1)} \max_{R:(k,l)\to (k',l')}\sup_{(t,x)\in [0,T]\times \mbx}\cb{\frac{\bar \rho^r(t,x,k)}{\bar \rho^r(t,x,k')},\frac{\bar \rho^r(t,x,l)}{\bar \rho^r(t,x,l')} }\le C_T.
\end{align}
From this bound, if we choose $r=r_N$, so that
\begin{align*}
r_N\searrow 0,\qquad \frac{1}{N}e^{r_N^{-d}T \Lambda_T\|\Phi\|_{L^\infty}}\to 0,\quad \mbox{ as }N\to\infty
\end{align*}
(specifically, $r_N\gg (\log N)^{-1/d}$), then the propagation of chaos holds for the dynamics $\{(\bX^{N,r_N}_t,\bXi^{N,r_N}_t)\}_{t\ge0}$, assuming $\mch_N(\rho_N^r\|(\bar\rho^r)^{\otimes N})(0) \leq O(N^{-1})$. In particular, if the initial $N$ particles are sampled independently from the law $\bar \rho^r(0)$, then  $\rho_N^r(0) = (\bar\rho^r)^{\otimes N}(0)$ and $\mch_N(\rho_N^r\|(\bar\rho^r)^{\otimes N})(0) = 0$. In this case, as $N\to\infty$ the empirical measure \eqref{emp-meas} converges to a solution $\bar \rho$ of the local 
chemical reaction-diffusion system \eqref{limitsystem} with $\Phi_R=\de_0$ (the Dirac delta measure at origin).
Of course, all the claims made here are based on the assumptions of \eqref{r-comp}, and the $L^1$ convergence of $\bar \rho^r\to \bar \rho$ as $r\searrow 0$.
We leave the justification of these assumptions and other details to the interested reader.
See also \cite{FLO19} for discussion of this issue in the context of different particle systems. 
\end{rem}

\section{Infinitesimal Generator of the Process $\{(\bX_t^N,\bXi_t^N)\}$} \label{generator}

The main objective of this section is to determine the infinitesimal generator $\bL_N$ of the process $\{(\bX_t^N,\bXi_t^N)\}_{t\ge 0}$, 
and subsequently, its adjoint operator $\bL_N^*$ with respect to the inner product $\langle \cdot,\cdot \rangle$ in $L^2(\Pi^N)$.
Throughout this section, $N\ge 1$ will be fixed.
For the sake of notational simplicity, we will suppress the subscript $N$ for $\by= \by_N=(\bx_N,\bxi_N)=(\bx,\bxi)\in \Pi^N$, $\rho=\rho_N$, 
and the superscript for $\bX_t=\bX^N_t,\bXi_t=\bXi^N_t$ in the computations involved.

\subsection{Generator of the process}

The generator of the process $(\bX_t,\bXi_t)$ has two components, corresponding to continuous diffusion (in $\bx$) and to jumps in the 
$\bxi$ coordinate (discrete change of type).  For $\varphi \in C_0^2(\Pi^N)$ ($C^2$ in $\bx$), let $\boldsymbol{\De}_N$ denote the spatial diffusion operator (in $\bx$) on $\Pi^N$:
\begin{align}\label{Ga-def}
\bDe_N(\varphi)(\by_N)&:=\f 12  \sum_{i=1}^N \si(\xi_i)^2 \De_{x_i}\varphi(\by_N).
\end{align}
Next, recall $\Ta^{ij}_R$ from \eqref{Ta-def}, and 
define linear operators $\bS_{N,R}^{ij},\bS_{N,R}, \bS_{N}:C_0(\Pi^N)\to C_0(\Pi^N)$
\begin{align}\label{S'-def}
(\bS_{N,R}^{ij}\varphi)(\by_N)&:=\chi_R^-(\xi_i,\xi_j)\Phi_R(x_i-x_j)[\varphi(\bx_N,\Ta_R^{ij}\bxi_N)-\varphi(\by_N)] \\
\bS_N&:= \sum_{R\in \fr}\bS_{N,R}:= \sum_{R\in \fr}  \sqb{\f 1 {N}\sum_{i=1}^N\sum_{j=1,j\neq i}^N \bS_{N,R}^{ij}}. \nonumber 
\end{align}

\begin{proposition} \label{prop:Gen}
	$(\bX_t^N,\bXi_t^N)$ defined by the $N$-particles system \eqref{SDE} is a Markov process on $\Pi^N$ with infinitesimal generator
	$\bL_N=\bDe_N+\bS_N$, where $\bDe_N,\bS_N$ are from \eqref{Ga-def}, \eqref{S'-def}. 
\end{proposition}

\begin{proof}

Fix a smooth function $\varphi \in C_0^2(\Pi^N)$ ($C^2$ in $\bx$), and consider $\varphi(\bX_t,\bXi_t)$.
Applying It\^{o}'s formula (see e.g. \cite[Chapter 9.3]{Klebaner12}) to the jump-diffusion equation \eqref{SDE} (see also \eqref{SDE-V}), we have
\begin{align}
\varphi(\bX_t,\bXi_t)&= \varphi(\bX_0,\bXi_0)+ \sum_{i=1}^N\int_0^t \f{\si(\Xi^i_s)^2}2 \De_{x_i} \varphi (\bX_s,\bXi_s) ds +M_t \nonumber \\
&\quad + \sum_{i=1}^N \sqb{\sum_{R\in\fr}\sum_{j\neq i} \int _0^t\sqb{\varphi(\bX_s,\Ta_R^{ij}(\bXi_{s-}))-\varphi(\bX_{s-},\bXi_{s-}) }  d\til E^{ij}_R(s)} , \nonumber \\
&=: \varphi(\bX_0,\bXi_0)+ G_t+ M_t +H_t \label{phi-eq}
\end{align}
where $M_t = \sum_{i=1}^N \int_0^t \sigma(\Xi^i_s) \nabla_{x_i} \varphi(\bX_s, \bXi_s) \cdot dB^i_s$ is a martingale with $M_0=0$. 
Now we take expectation for each term of the identity above.
Since $\rho$ is the law of the process $(\bX_t,\bXi_t)$, the expectation of the left hand side is given by
\begin{align*}
\mbe \varphi(\bX_t,\bXi_t) =\int_{\Pi^N} \varphi(\by) \rho(t,d\by )=\inn{\varphi,\rho(t)}.
\end{align*}
The above also holds for the first term from the right hand side of \eqref{phi-eq} (with $t=0$).
As for the second term $G_t$ of \eqref{phi-eq}, it again follows by the definition of $\rho$ that
\begin{align*}
\mbe G_t&=  \int_0^t \int_{\Pi^N} \sum_{i=1}^N \frac{1}{2} \si(\xi_i)^2 \De_{x_i} \varphi(\by) \rho(s,d\by) ds \\
&= \int_0^t \int_{\Pi^N}(\bDe_N\varphi)(\by)\rho(s,d\by)ds
= \int_0^t \inn{\bDe_N\varphi, \rho(s) }ds .
\end{align*}

Now consider the expectation of the last term $H_t$ in \eqref{phi-eq}, which is given by
\begin{align*}
H_t&=\sum_{R\in \fr} \sum_{i,j=1,i\neq j}^N L_R^{ij}(t),\qquad L_R^{ij}(t):= \int _0^t\sqb{\varphi(\bX_s,\Ta_R^{ij}(\bXi_{s-}))-\varphi(\bX_{s-},\bXi_{s-}) } d\til E^{ij}_R(s).
\end{align*}
Consider $L_R^{ij}(t)$ for a fixed $1\le i,j\le N$ with $i\neq j$ and reaction $(k,l)\xrightarrow{R} (k',l')$ from $\fr$.
Observe that the integrand $\varphi(\bX_{s},\Ta^{ij}_R(\bXi_{s-}))-\varphi(\bX_{s},\bXi_{s-})$ is left-continuous (and hence predictable),
and the integrator $\{\til E^{ij}_R(s)\}_{s\ge 0}$, by \eqref{SDE}, has compensator
\begin{align*}
A^{ij}_R(s)&= \f 1 {N} \int_0^t \chi^-_{R}(\Xi_{s-}^i,\Xi_{s-}^j) \Phi_R(X_s^i-X_s^j)ds.
\end{align*}
Therefore, the expectation of the process $L_R^{ij}(t)$ equals to that of the 
integrand $\varphi(\bX_{s},\Ta_R^{ij}(\bXi_{s-}))-\varphi(\bX_{s},\bXi_{s-})$ integrating against the compensator $A_R^{ij}(s)$. That is,
\begin{align*}
\mbe L_R^{ij}(t)&= \mbe \int_0^t[\varphi(\bX_{s},\Ta_R^{ij}(\bXi_{s-}))-\varphi(\bX_{s}, \bXi_{s-}) ] dA^{ij}_R(s) \\
&=\f 1 {N} \mbe \int_0^t \chi^-_R (\Xi_{s-}^i,\Xi_{s-}^j) [\varphi(\bX_{s},\Ta_R^{ij}(\bXi_{s-}))-\varphi(\bX_{s}, \bXi_{s-}) ]\Phi_R(X_s^i-X_s^j)ds \\
&= \f 1 {N} \int_0^t \int_{\Pi^N} \chi^-_R(\xi_i,\xi_j)\Phi_R(x_i-x_j)[\varphi(\bx,\Ta_R^{ij}(\bxi))-\varphi(\bx,\bxi)]\rho(s,d\by)ds.
\end{align*}

Summing up $\mbe L^{ij}_R(t)$ for $R\in \fr$, and $1\le i,j\le N$ with $i\neq j$, we find that $\mbe H_t$ from \eqref{phi-eq} is given by
\begin{align*}
\mbe H_t&= \sum_{ R\in \fr}\sum_{i,j=1,i\neq j}^N \mbe L^{ij}_R(t)= \f 1 N\sum_{R\in \fr} \sum_{i\neq j}\int_0^t\inn{\bS_{N,R}^{ij} \varphi, \rho(s)}  ds = \int _0^t \inn{\bS_N\varphi,\rho(s)}ds  .
\end{align*}

Combining all the computations for the expectation of \eqref{phi-eq}, we conclude
\begin{align}\label{weak-form}
\inn{\varphi,\rho(s)}\big|_{s=0}^{t} &= \int_0^t \inn {\boldsymbol{\mcll}_N\varphi,\rho(s)} ds ,\quad \boldsymbol{\mcll} _N:=\bDe_N+\bS_N,
\end{align}
which shows that the generator of $\{(\bX_t,\bXi_t)\}_{t\ge 0}$ is $\bL_N$.
\end{proof}

\subsection{The adjoint equation}
Now we turn our attention to the adjoint operator of $\bL_N$. The main result of this subsection is stated as follow.
\begin{proposition} \label{prop:Genadj}
The adjoint of generator $\bL_N$ in the $L^2(\Pi^N)$-sense is given by $\bL_N^*=\bDe_N+\bS_N^*$, where
\begin{align}
\bS_N^*(\psi)(\by_N)&= \f 1 N \sum_{R\in \fr} \sum_{\substack{i,j=1\\i\neq j}}^N  \Phi_R(x_i-x_j)\sqb{ \chi_R^+(\xi_i,\xi_j)\psi(\bx_N,\til\Ta_{R}^{ij}\bxi_N)-\chi_R^-(\xi_i,\xi_j)\psi(\by_N) } .\label{S-def}
\end{align}
The law $\rho_N(t)$ of the process is the unique strong solution to the Fokker-Planck equation with initial data $\rho_N(0)\in L^1(\Pi^N)$:
\begin{align}\label{FP-eq}
\partial _t \rho_N &= \bL_N^*(\rho_N),\qquad (t,\by_N)\in (0,\infty)\times \Pi^N.
\end{align}
Moreover, $\rho_N$ has the following regularity: for any $p\in [1,\infty)$
\begin{align}\label{b-valued}
\rho_N\in C([0,\infty);L^1(\Pi^N))\cap C((0,\infty);W^{2,p}(\Pi^N))\cap C^1((0,\infty);L^p(\Pi^N)).
\end{align}
If additionally $\rho_N(0)\in W^{2,\infty}(\Pi^N)$, then $\rho_N\in C([0,\infty);W^{2,p}(\Pi^N))\cap  C^1([0,\infty);L^p(\Pi^N))$.
\end{proposition}

\begin{proof}
Observe that $\bDe_N$ is self-adjoint (w.r.t. inner product $\inn{\cdot,\cdot}$ in $L^2(\Pi^N)$), and thus $\bDe_N^*=\bDe_N$.  We now compute the dual of the operator $\bS_N$.
Recall the definition from \eqref{S'-def}.
	To determine the dual of $\bS_N$, we begin with that of $\bS_{N,R}^{ij}$.
	Fix a pair of functions $\varphi,\psi \in C_0(\Pi^N)$.
	Then
	\begin{align*}
	\inn {\bS_{N,R}^{ij}\varphi,\psi}
	 = \int_{\Pi^N} \chi^-_R(\xi_i,\xi_j)\Phi_R(x_i-x_j)[\varphi(\bx,\Ta_R^{ij}\bxi)-\varphi(\by)]\psi(\by)dm^N(\by).
	\end{align*}
	From here, we claim that the following identity holds:
	\begin{align*}
	&\int_{\Pi^N} \chi^-_R(\xi_i,\xi_j)\Phi_R(x_i-x_j)\varphi(\bx,\Ta_R^{ij}\bxi)\psi(\by) dm^N(\by)\\
	&= \int_{\Pi^N} \chi^+_R(\xi_i,\xi_j)\Phi_R(x_i-x_j)\varphi(\by)\psi(\bx,\til \Ta_R^{ij}\bxi) dm^N(\by),
	\end{align*}
	Indeed, to verify this identity, by permuting the indices it suffices to check it for $(i,j)=(1,2)$.
	Write $\bxi=(\xi_1,\xi_2,\hat \bxi)$, with $\hat \bxi=(\xi_3,\xi_4,\cdots, \xi_N)\in \fs^{N-2}$.
	Then by the definitions of $\chi_R^-,\Ta_R^{12},\til \Ta_R^{12}$ (see \eqref{chi-def}, \eqref{Ta-def}, \eqref{Tta-def}),
	and recall also $R\in \fr$ is so that $(k,l)\xrightarrow{R}(k',l')$, we have
	\begin{align*}
	&\int_{\Pi^N} \chi_R^-(\xi_1,\xi_2)\Phi_R(x_1-x_2)\varphi(\bx,\Ta_R^{12}\bxi)\psi(\by,\bxi)dm^N(\by)\\
	&= \sum_{\hat \bxi\in \fs^{N-2}} \int_{\mbx^N}\Phi_R(x_1-x_2)\varphi(\bx,k',l',\hat \bxi)\psi(\bx,k,l,\hat \bxi)d\bx\\
	&= \int_{\Pi^N} \chi_R^+(\xi_1,\xi_2)\Phi_R(x_1-x_2)\varphi(\bx,\bxi)\psi(\bx,\til \Ta_{R}^{12}\bxi)dm^N(\by).
	\end{align*}
	Using the identity we just established, it follows
	\begin{align*}
	\inn{\bS_{N,R}^{ij}\varphi,\psi} &= \int_{\Pi^N} \Phi_R(x_i-x_j)\varphi(\bx,\bxi)\sqb{\chi_R^+(\xi_i,\xi_j)\psi(\bx,\til \Ta_R^{ij}\bxi)-\chi_R^-(\xi_i,\xi_j)\psi(\by)} dm^N(\by).
	\end{align*}
	Thus,
	\begin{align*}
	(\bS_{N,R}^{ij})^*\psi(\by)&= \Phi_R(x_i-x_j)\sqb{\chi_R^+(\xi_i,\xi_j)\psi(\bx,\til \Ta_R^{ij}\bxi)-\chi_R^-(\xi_i,\xi_j)\psi(\by)}.
	\end{align*}
	Summing these operators up for $1\le i,j\le N$ with $i\neq j$, and $R\in\fr$ yields \eqref{S-def}.
	
	Existence, uniqueness, and regularity of solutions to \eqref{FP-eq} follows by the standard theory of parabolic systems. 
	In order not to disrupt the flow of presentation, we postpone this part of the proof to the appendix (Proposition \ref{well-posed}).
	By the theory of Markov processes, the law $\rho_N(t)$ of the process satisfies the forward equation \eqref{FP-eq} \cite{EK86}. 
\end{proof}

Despite being straightforward from our construction, for the sake of completeness we give the proof to the preservation of
exchangeability for the dynamics \eqref{SDE} before ending this subsection. 

\begin{lemma} \label{lem:exch}
If $\rho_N(0) \in \mathcal{P}_{sym}(\Pi^N)$, then $\rho_N(t) \in \mathcal{P}_{sym}(\Pi^N)$ for all $t > 0$. 
Equivalently, the particles $\bY^N_t := (\bX^N_t, \bXi^N_t)$ are exchangeable for any $t> 0$ if the initial particles $\bY^N_0$ are exchangeable.
\end{lemma}

\begin{proof}
Lemma \ref{lem:exch} is an immediate consequence of Proposition \ref{prop:Genadj}, the uniqueness of the
PDE \eqref{FP-eq} correspondent to the initial data,  and the symmetry of the operator $\bS_N^*$ from \eqref{S-def}.
Specifically, for any permutation $\tau$ on $\{1,\dots,N\}$ the operator $\bS_N^*$ satisfies
\begin{align}\label{Sadjsym}
\tau(\bS_N^*\psi )=\bS_N^*(\tau \psi)
\end{align}
Here we are using $\tau \psi$ to denote the action of $\tau$ on a function $\psi:\Pi^N\to \mbr$ by $(\tau \psi)(\by_N)=\psi(\tau\by_N)$, where $\tau \by_N$ denote the state variables with permuted indices:
\[
\tau \by_N = \tau \left((x_1, \xi_1), \cdots, (x_N,\xi_N) \right) = \left( (x_{\tau(1)}, \xi_{\tau(1)}), \cdots, (x_{\tau(N)}, \xi_{\tau(N)})\right),
\]

To see why the symmetry property \eqref{Sadjsym} holds, observe that
\begin{align}
	&\qquad \tau (\bS_N^*\psi )(\by)= \bS_N^*(\psi)(\tau \by) \nonumber \\
	&= \f 1 N \sum_{R\in \fr} \sum_{\substack{i,j=1\\i\neq j}}^N  \Phi_R(x_{\tau(i)}-x_{\tau(j)})\Big[ \chi_R^+(\xi_{\tau(i)},\xi_{\tau(j)})\psi(\tau \bx,\til\Ta_{R}^{ij}\tau \bxi) -\chi_R^-(\xi_{\tau(i)},\xi_{\tau(j)})\psi(\tau \by) \Big] ,\label{Stau}
\end{align}
where here we suppress the subscript $N$ for $\by,\bx,\bxi$.
Recalling the definition \eqref{Tta-def}, it is easy to see that for any pair of indices $(i,j)$ with $i \neq j$, and if $(m,\ell) = (\tau^{-1}(i),\tau^{-1}(j))$, then
\begin{align*}
\til\Ta_{R}^{ij}\tau \bxi = \tau \til\Ta_{R}^{m,\ell} \bxi. 
\end{align*}
Consequently, \eqref{Stau} is equivalent to
\begin{align}
	\tau (\bS_N^*\psi)(\by)&= \f 1 N \sum_{R\in \fr} \sum_{\substack{m,\ell=1\\m\neq \ell}}^N  \Phi_R(x_{m}-x_{\ell})\sqb{ \chi_R^+(\xi_{m},\xi_{\ell})\psi(\tau \bx,\tau \til\Ta_{R}^{m \ell } \bxi)-\chi_R^-(\xi_{m},\xi_{\ell})\psi(\tau \by) } \nonumber \\
& =  \bS_N^*(\tau \psi)(\by). \qedhere 
\end{align}
\end{proof}

\subsection{Formal derivation of the mean field limit}\label{sec:formalmfl}
We now formally derive the mean field limit equation \eqref{MFE} using the result established earlier; a rigorous justification of the mean field limit is carried out in the next section. 

If $\mu_N(t)\in \mathcal{P}(\Pi)$ is the empirical measure \eqref{emp-meas} and $\phi \in C_C^\infty(\Pi)$ is any test function, then we claim that
\begin{align}\label{exp-weak-form}
\mbe \inn{\phi,\mu_N(s)}\Big|_{s=0}^t&= \mbe \sqb{\int_0^t\rb{\f 12 \inn{\si(\xi)^2 \De \phi,\mu_N} + \inn{\phi,\bar T\mu_N} }(s) ds } + O(N^{-1}),
\end{align}
where for the rest of this section $\inn{\cdot,\cdot}$ denotes the bilinear form on $C_0(\Pi)\times \mathcal{M}(\Pi)$ (i.e.  \eqref{bilinear} with $N=1$). Recall that the nonlinear operator $\bar T$ is defined at \eqref{T-def}. Therefore, if $\mu_N(t) \rightarrow \bar \rho(t)dm$ in the appropriate sense as $N \to \infty$, where $\bar \rho(t)\in C(\Pi)$ is some smooth deterministic function,
then formally passing to the limit in \eqref{exp-weak-form} we obtain the weak formulation of \eqref{MFE}:
\begin{align*}
\inn{\phi,\bar \rho(s)}\Big|_{s=0}^t &= \int _0^t \sqb{\f 12 \inn{\si(\xi)^2\De \phi,\bar \rho(s)}+ \inn{\phi,\bar T\bar \rho(s)}} ds \quad \mbox{for all }\phi \in C^\infty_C(\Pi) .
\end{align*}

The relation \eqref{exp-weak-form} may be derived as follows. From \eqref{weak-form} we know that for any test function $\varphi \in C_0^2(\Pi^N)$,
$$
\mbe \varphi(\bX_s,\bXi_s)\Big|_{s=0}^t  = \int_0^t \mbe  \boldsymbol{\mcll}_N\varphi(\bX_s,\bXi_s) ds
= \int_0^t \mbe  (\bDe_N+\bS_N) \varphi(\bX_s,\bXi_s) ds.
$$
In particular, if we choose $\varphi$ of the form $\varphi (\bx,\bxi)= N^{-1} \sum_{i=1}^N \phi (x_i, \xi_i)$ with $\phi \in C_C^\infty(\Pi)$, then
the left side reads
\begin{align}\label{eq:mfl1}
\mbe \varphi(\bX_s,\bXi_s)\Big|_{s=0}^t = \mbe \int_{\Pi} \phi(x,\xi) \mu_N(s,dx,  d\xi)\Big|_{s=0}^t = \mbe \inn{\phi,\mu_N(s)}\Big|_{s=0}^t.
\end{align}
Moreover,
\begin{align}
\mbe \bDe_N\varphi(\bX_s,\bXi_s) &=\frac{1}{2} \mbe \sqb{\frac{1}{N} \sum_{i=1}^N  \sigma(\Xi^i_s)^2 \Delta_{x_i} \phi(X_s^i, \Xi_s^i)} \nonumber\\
& = \mbe\sqb{\int_{\Pi} \frac{1}{2}\sigma(\xi)^2\Delta_{x}\phi(x) \mu_N(s, dx, d\xi) }
= \f 12 \mbe \inn{\si(\xi)^2\De\phi,\mu_N(s)}
.\label{eq:mfl2}
\end{align}
Consider $\mbe [\bS_N\varphi(\bX_s,\bXi_s)]$. Recall the definition of $\bS_N=\sum_{R\in \fr}\bS_{N,R}$ from \eqref{S'-def}, 
so consider $\mbe [(\bS_{N,R}\varphi)(\bX_s,\bXi_s)]$ for a fixed $R:(k,l)\to(k',l')$ from $\fr$.
By the definition of $\Ta_R^{ij}$,
\begin{align}
& \mbe[ (\bS_{N,R} \varphi)(\bX_s,\bXi_s)]\nonumber  \\
&=\mbe\sqb{\f 1 N \sum_{i=1}^N\sum_{j\neq i}^N \chi_R^-(\Xi^i_s,\Xi^j_s)\Phi_R(X^i_s-X^j_s)[\varphi(\bX^N_s,\Ta_R^{ij}(\bXi^N_s))-\varphi(\bX^N_s,\bXi_s^N)]}\nonumber \\
& =\mbe\sqb{   \f 1 {N^2} \sum_{\substack{i,j=1,i\neq j}}^N  \chi_R^-(\Xi_s^i,\Xi_s^j)\Phi_R(X_s^i-X_s^j)[\phi(X_s^i, k^\prime) - \phi(X_s^i, \Xi_s^i) + \phi(X_s^j, l^\prime) - \phi(X_s^j, \Xi_s^j)] } \nonumber \\
&=\mbe[I^R_1-I^R_2+I^R_3-I^R_4]
.\label{eq:mfl3}
\end{align}
Consider $I_1=I_1^R$.
Let us include the diagonal terms $(i=j)$, which is at most $CN^{-1}$, with constant $C$ depending on $\|\Phi_R\|_{L^\infty}$, $\phi$, into the double summation $I_1$.
Noting also by definition that $\chi_R^-(\xi_i,\xi_j) = \chi_k(\xi_i) \cdot \chi_l(\xi_j)$, one then has
\begin{align*}
I_1+O(N^{-1})&= \f 1 {N^2} \sum_{i,j=1}^N\ \chi_R^-(\Xi_s^i,\Xi_s^j)\Phi_R(X_s^i-X_s^j)\phi(X_s^i, k^\prime) \\
&= \f 1{N^2}\sum_{\xi=1}^n\sum_{i,j=1}^N\chi_{k'}(\xi)\chi_{k}(\Xi_s^i)\chi_{l}(\Xi_s^j)\Phi_R(X_s^i-X_s^j)\phi(X_s^i,\xi)\\
&= \f 1{N}\sum_{\xi=1}^n\sum_{i=1}^N \chi_{k'}(\xi)\chi_k(\Xi_s^i)\phi(X_s^i,\xi)(\Phi_R*\mu_N)(s,X^i_s,l)\\
&= \sum_{\xi=1}^n  \int_{\mbx} \chi_{k'}(\xi) \phi(x,\xi)(\Phi_R*\mu_N)(s,x,l)d\mu_N(s,dx,k) \\
&= \int_{(x,\xi)\in\Pi} \chi_{k'}(\xi)\phi(x,\xi)(\Phi_R*\mu_N)(s,x,l)\mu_N(s,dx,k)d\#(\xi),
\end{align*}
($\#$ denotes the counting measure on $\fs$) where in obtaining the second and  third equality above we used the fact that
$$
\phi(x, k^\prime) = \sum_{\xi=1}^n\chi_{k'}(\xi)\phi(x, \xi),\qquad \sum_{j=1}^N \chi_l(\Xi_s^j)\Phi_R(X_s^i-X_s^j)=(\Phi_R*\mu_N)(s,X_s^i,l).
$$
In the same manner, one can derive similar expressions for the remaining three terms in  \eqref{eq:mfl3}, specifically for $I_3$ we have
\begin{align*}
I_3+O(N^{-1})=\int_{(x,\xi)\in\Pi} \chi_{l'}(\xi)\phi(x,\xi)(\Phi_R*\mu_N)(s,x,k)\mu_N(s,dx,l)d\#(\xi).
\end{align*}
 Note that to derive the expression above for $I_3$ 
 we also used the symmetry of the kernel $\Phi_R$ defined in \eqref{Phi-def}, without which the convolution
above would have been replaced by $\Phi_R^- \ast \mu_N$ where $\Phi_R^-(\cdot) = \Phi_R(-\cdot)$. We make this symmetry assumption 
to simplify expressions of the mean field limit \eqref{limitsystem}. 
Summing up $I_1,I_3$, and using the notation $\bar T_R^+$ from \eqref{TR-def}, we have
\begin{align*}
I_1+I_3&= O(N^{-1})+ \int_{\Pi} \phi(x,\xi)\bar T_R^+ \mu_N(dx,d\xi)=O(N^{-1})+\inn{\phi,\bar T_R^+\mu_N}.
\end{align*}
Here, $\bar T_R^\pm$ from \eqref{TR-def}, initially defined on $C_0(\Pi)\to C_0(\Pi)$, can be extended naturally to 
a map from the space of signed measures on $\Pi$ to itself (as $\Phi_R$ is bounded integrable).
The same computation for $I_2,I_4$ gives
\begin{align*}
-(I_2+I_4)=O(N^{-1})+\inn{\phi ,\bar T_R^-\mu_N},
\end{align*}
and hence by the definition $\bar T_R=\bar T_R^-+\bar T_R^+$ (see \eqref{T-def})
\begin{align*}
(I_1+I_3)-(I_2+I_4)&= O(N^{-1})+\inn{\phi, T_R\mu_N}.
\end{align*}
Summing up $R\in \fr$ in \eqref{eq:mfl3} and recalling $\bar T$ from \eqref{T-def}, we arrive at
\begin{align*}
\mbe \sqb{(\bS_{N}\varphi)(\bX_s,\bXi_s)}&= \sum_{R\in \fr} \mbe[I_1^R-I_2^R+I_3^R-I_4^R]\\
&= O(N^{-1}) + \mbe \sqb{\sum_{R\in \fr} \inn{\phi, \bar T_R\mu_N} }
= O(N^{-1}) + \mbe \sqb{\inn{\phi, \bar T\mu_N} }.
\end{align*}
Combining this with \eqref{eq:mfl1}, \eqref{eq:mfl2}, we have shown \eqref{exp-weak-form}.  This concludes  the formal derivation of the mean field limit equation.


\section{Proof of the Main Result (Theorem \ref{main})} \label{Int-ineq}
To prove Theorem \ref{main}, we will compare the joint distribution of $N$ particles with the tensorized 
law $\bar \rho_N:=\bar \rho^{\otimes N}$ of the mean field limit in terms of their relative entropy. Recall that the 
joint distribution $\rho_N(t)$ of the process $\{(X^i_t,\Xi_t^i)_{i=1,\cdots, N}\}_{t\ge 0}$, by Proposition \ref{prop:Genadj},
satisfies the Fokker-Planck equation \eqref{FP-eq}.  Before we estimate the relative entropy, let us first identify the PDE satisfied by the tensorized law $\bar \rho_N$.
Recall that the mean field limit $\bar \rho$ satisfies the system \eqref{MFE} with initial data \eqref{initial} (see Proposition \ref{well-posed} for the existence and uniqueness result). 
The tensorized law
\begin{align*}
\bar \rho_N(t,\by_N)&= \prod_{i=1}^N \bar \rho(t,y_i),\quad \by_N=(y_1,\cdots, y_N)=((x_1,\xi_1),\cdots,(x_N,\xi_N))\in \Pi^N,
\end{align*}
then solves the PDE system
\begin{align}\label{ts-eq}
\partial _t \bar \rho_N &= \bDe_N\bar \rho_N + \bar \bT_N\bar \rho_N,\qquad (t,\by_N)\in (0,\infty)\times \Pi^N,
\end{align}
where $\bDe_N$ is the diffusion operator from \eqref{Ga-def}, and $\bar \bT_N$ is as follows, with $\bar T$ from \eqref{T-def}:
\begin{align}\label{barT-def}
(\bar \bT_N\bar \rho_N)(\by_N)&:= \sum_{i=1}^N \sqb{\bar T(\bar \rho)(y_i)\prod_{j=1,j\neq i}^N \bar\rho(y_j)}.
\end{align}

Recall the \emph{normalized relative entropy} of $\rho_N$ with respect to $\bar \rho_N$:
\begin{align*}
W_N(t)=\mch_N(\rho_N\|\bar \rho_N)(t)&:= \f1 N\int_{\Pi^N} \rho_N(t)\log\rb{\f{\rho_N(t)}{\bar \rho_N(t)}} dm^N.
\end{align*}
Our main objective is to establish a differential inequality for this quantity, then invoke Gr\"onwall lemma to obtain an estimate in terms of $N$ and $W_N(0)$.
Taking the time derivative of $W_N(t)$ and using equations \eqref{FP-eq}, \eqref{ts-eq}, we have for all $t>0$ that
\begin{align}\label{RE-ineq}
W_N'(t)&= \f 1 N \int_{\Pi^N} \sqb{\partial _t \rho_N \log\rb{\f{\rho_N}{\bar \rho_N}}+ \rho_N\rb{\f{\partial _t\rho_N}{\rho_N}-\f{\partial _t \bar \rho_N}{\bar \rho_N} } }dm^N
= D(t)+G(t), 
\end{align}
where
\begin{align}
D(t)&= \f 1 N\int _{\Pi^N} \sqb{\bL^*_N(\rho_N) \log\f{\rho_N}{\bar \rho_N} + \rho_N\rb{\f{\bL ^*_N(\rho_N)}{\rho_N} -\f{\bL^*_N(\bar \rho_N)}{\bar \rho_N}} } dm^N, \label{D-def}\\
G(t)&= \f 1 N \int _{\Pi^N} \rho_N \f{(\bar \bT_N-\bS^*_N)(\bar \rho_N)}{\bar \rho_N}  dm^N. \nonumber
\end{align}
All functions $\rho_N,\bar \rho_N$ involved above are evaluated at time $t> 0$.
The calculations above can be rigorously justified since the solutions have sufficient regularity, namely, for any $p\in [1,\infty)$, 
\begin{align*}
\rho_N,\bar \rho_N&\in C([0,\infty);L^1 (\Pi^N))\cap C((0,\infty); W^{2,p}(\Pi^N))\cap C^1((0,\infty);L^p(\Pi^N)),
\end{align*}
and $\rho_N, \bar \rho_N > 0$ are bounded away from zero when $t > 0$. (See Propositions \ref{prop:Genadj}, \ref{well-posed} and \ref{positive}.) 
We proceed with estimating the quantities $D(t)$ and $G(t)$.

\subsection{Estimating $D(t)$, Dissipation of Relative Entropy}
$D(t)$ is in fact nonpositive, due to the diffusive nature of the operator $\bL_N^*$.
Specifically, it is due to the following lemma.

\begin{lemma}\label{lem-REineq} 
	Let $\rho_N,\bar \rho_N$ be given in Theorem \ref{main}, with $\mch(\rho_N\|\bar \rho_N)(0)<\infty$. 
	Then the integral \eqref{D-def} is finite and nonpositive for all $t>0$.

\end{lemma}



\begin{proof}[Proof of Lemma \ref{lem-REineq}]
	The strict positivity condition for $\bar \rho_0$ in the main theorem is not essential in the proof, and we will prove the lemma with the weaker assumption of $\bar \rho_0$ being propagated in the sense defined before Proposition \ref{positive}.
	We will first show the following integral is finite and nonpositive, for any appropriate pair of densities $\rho,\til \rho\in L^1(\Pi^N)$:
	\begin{align}\label{D-def2}
	\int_{\Pi^N}  \sqb{\bL_N^*(\rho)\log\rb{\f{\rho}{\til \rho}} + \rho\rb{\f{\bL_N^*(\rho)}{\rho}-\f{\bL_N^*(\til\rho)}{\til \rho}}}dm^N \le 0. 
	\end{align}
	Let $\{u_t,\til u_t\}_{t\ge 0}$ be the solutions of \eqref{FP-eq}, with initial data $u_0=\rho, \til u_0=\til \rho$, namely,
	\begin{align}\label{FP-eq2}
	\partial _t u_t =\mcll^* (u_t),\, u_0=\rho,\qquad \partial _t \tilde{u}_t = \mcll^*(\tilde{u}_t),\, \tilde{ u}_0=\tilde{\rho}.
	\end{align}
	The proof of \eqref{D-def2} then breaks into two parts: first, we show that the relative entropy between two solutions $W(t)=\mch(u_t\|\til u_t)$ 
	is non-increasing w.r.t. the time variable $t\ge 0$, then show that the integral in \eqref{D-def2} is the derivative of $W(t)$ evaluating at $t=0$, and thus must be non-positive. 
	The precise condition for $\rho,\til\rho$ will be given later.
	
	We start with verifying the claim that $W(t)$ is non-increasing, for any initial data $\rho,\til\rho\in L^1(\Pi^N)$ with $W(0)=\mch(\rho\|\til \rho )<\infty$. 
	In fact, this monotonicity property of relative entropy, often referred to as the {\em data processing inequality} \cite{Wilde17}, 
	is well-known in physics and information theory community, and we will provide a short proof to it. 	Let $(\mathcal{X},\mu)$ be a measure space and consider two probability densities $p(x,y),\til p(x,y)\in L^1(\mathcal{X}^2,\mu^{\otimes 2})$. 
	Denote by $p(y|x),\tilde{p}(y|x)$ their conditional densities, and by $p(x), \tilde{p}(x)$ the $x$-marginals. 
	With a slight abuse of notation, we denote the \emph{averaged conditional relative entropy} between  $p(y|x)$ and $\tilde{p}(y|x)$  by
	\begin{align*}
	\mch(p(y|x)\|	\tilde{p}(y|x)) &= \int_\mathcal{X} p(x) \int_{\mathcal{X}} p(y|x)\log\frac{p(y|x)}{\tilde{p}(y|x)} d\mu(y) d\mu(x).
	\end{align*}
	By the chain rule \cite[Theorem C.3.1]{DE97} of the relative entropy, we have
	\begin{align*}
	\mch(p(x,y)\|\tilde{p}(x,y))&= \mch(p(x)\|\tilde{p}(x))+ \mch(p(y|x)\|\tilde{p}(y|x)).
	\end{align*}
	
	Now return to the solutions $\{u_t,\til u_t\}_{t\ge 0}$ defined earlier, at \eqref{FP-eq2}. 
	Recall from Proposition \ref{prop:Gen} that they are probability densities of the Markov process $\{\bY_t=(\bX_t,\bXi_t)\}_{t\ge 0}$ defined by \eqref{SDE}, 
	with initial distribution $\rho,\tilde{\rho}$ respectively.
	Given $t,s\ge 0$, let $u_{t,s},\tilde{u}_{t,s}$ denote the joint probability density of $(\bY_t,\bY_s)$.
	Similarly, let $u_{t|s},\tilde{u}_{t|s}$ be the conditioned densities of $\bY_t$  given $\bY_s$.
	By the chain rule, for any $t,h\ge 0$,
	\begin{align*}
	\mch(u_t\|\tilde{u}_t)+ \mch(u_{t+h|t}\|\tilde{u}_{t+h|t})= \mch(u_{t,t+h}\|\tilde{u}_{t,t+h})
	= \mch(u_{t+h}\|\tilde{u}_{t+h}) + \mch(u_{t|t+h}\|\tilde{u}_{t|t+h}).
	\end{align*}
	Since $u,\tilde{u}$ are defined by the same Markov process (with the same transition kernel), their condition densities coincide, 
	i.e. $u_{t+h|t}=\tilde{u}_{t+h|t}$. As a result, the second term of the left hand side vanishes. By the non-negativity of relative entropy, it follows from the last equation that
	\begin{align*}
	W(t+h)=\mch(u_{t+h}\|\tilde{u}_{t+h})\le \mch(u_t\|\tilde{u}_t)=W(t).
	\end{align*}
	Therefore, $t\mapsto W(t)$ is nonincreasing.
	
	We next proceed with verifying that the integral \eqref{D-def2} is given by $W'(0)$. 
	Indeed, formally differentiating $W(t)$ respect to $t$ and using \eqref{FP-eq2} and nonincreasing of $W(t)$ yield
	\begin{align*}
	W'(0)&= \frac{d}{dt}\Big|_{t=0} \int_{\Pi^N} u_t\log\rb{\f{u_t}{\tilde{u}_t} }dm^N\\
	& = \int_{\Pi^N} \sqb{\bL_N^*(\rho)\log\rb{\f{\rho}{\tilde{\rho}} }+ \rho \rb{\f{\bL_N^*(\rho)}{\rho} -\f{\bL_N^*(\tilde{\rho})}{\tilde{\rho}} }  }dm^N\le 0.
	\end{align*}
	In conclusion, if $\rho,\til \rho\in L^1(\Pi^N)$ are so that $t\mapsto W(t)=\mch(u_t\|\til u_t)$ is (right-hand) differentiable at $t=0$, then \eqref{D-def2} holds. 
	
	Now let us apply \eqref{D-def2} to conclude Lemma \ref{lem-REineq}.
	Fix $t_0>0$ and consider the integral $D(t_0)$ from \eqref{D-def}. 
	Using \eqref{D-def2} with $\rho_N(t_0),\bar \rho_N(t_0)$ taking the roles of $\rho,\til\rho$, we then have $D(t_0)\le 0$, provided that the correspondent $W(t)$ is differentiable at $t=0$. 
	So it remains to check this differentiability condition. 
	
	First, from the definition, $\til u_t$ is the solution of \eqref{FP-eq} with initial data $\til u_0=\bar \rho_N(t_0)=\bar \rho^{\otimes N}(t_0)$, 
	where $\bar \rho$ is the solution of \eqref{MFE}.
	By Propositions \ref{well-posed} and \ref{positive} (and under the assumptions of Theorem \ref{main}), 
	we know that $\bar \rho(t_0)\in W^{2,p}(\Pi)$ for all $p\in[1,\infty)$, $\inf_{\Pi} \bar \rho(t_0)>0$, 
	and hence $\bar \rho_N(t_0)\in W^{2,p}(\Pi^N),\inf_{\Pi^N} \bar\rho_N(t_0)>0 $. Thus, by Propositions \ref{well-posed} 
	and \ref{positive} again, we have for any $p\in [1,\infty)$ and $\de>0$ that
	\begin{align*}
	\til u_t\in C^1([0,\infty);L^p(\Pi^N)), \quad \inf_{(t,\by)\in [0,\de]\times \Pi^N} \til u_t(\by)>0.
	\end{align*}
	Next, since $u_t$ is the solution of \eqref{FP-eq} with $u_0=\rho_N(t_0)$, we have then $u_t=\rho_N(t_0+t)$. 
	By the regularity of solutions (Proposition \ref{well-posed}), we have $u_t\in C^1([0,\infty);L^p(\Pi^N))$ for any $p\in [1,\infty)$.
	Moreover, by Proposition \ref{positive}, the function $u_t$ has the following property: for any $\bxi\in \fs^N$ 
	\begin{align*}
	\mbox{either }u_t(\bx,\bxi)\equiv 0\;\;\forall (t,\bx)\in [0,\infty)\times \mbx^N,\quad \mbox{ or }\quad \inf_{(t,\bx)\in[0,\de]\times \mbx^N} u_t(\bx,\bxi)>0\;\;\forall \de>0.
	\end{align*}
	Let $\mathcal{A}\subset \fs^N$ be the set of all $\bxi$ so that the latter of the above holds. Then 
	\begin{align*}
	W(t)&= \int_{\Pi^N} \chi_{\mathcal{A}}(\bxi) u_t (\log u_t -\log \til u_t) dm^N.
	\end{align*}
	Since $u_t,\til u_t\in C^1([0,\infty);L^p(\Pi^N))$ are uniformly bounded above, and below from $0$ over 
	a time interval $[0,\de]\times \mbx^N$ on the set $\mathcal{A}$, this follows $\chi_{\mathcal{A}}\log(\f{u_t}{\til u_t})\in C^1([0,\de);L^p(\Pi^N))$ for any $p\in[1,\infty)$.
	This implies $W(t)$ is differentiable at $t=0$, which finishes the proof of Lemma \ref{lem-REineq}.	
\end{proof}

\subsection{Estimating $G(t)$}
Next consider $G(t)$ from \eqref{RE-ineq}.
In the coming computation, we suppress the time variable $t$ for $\rho_N,\bar \rho_N$, and the subscript $N$ for $\by_N,\bx_N,\bxi_N, \bar \bT_N, \bS_N^*$ and so on, since $N$ will be fixed.
To further ease our notation, we also introduce functions $(u_1,\cdots, u_n)=\bar \rho$, that is,
\begin{align*}
u_\xi(t,x)=\bar \rho(t,x,\xi)\quad \mbox{for }\xi\in \fs.
\end{align*}
Recall that these functions satisfy the mean field limit equation \eqref{MFE}.
Recall also the quantity $G(t)$ is the expectation, against $\rho_N$, of the function $\bar \rho_N^{-1}(\bT-\bS^*)\bar \rho_N$.
We begin with considering $\bar \rho_N^{-1}\bar \bT\bar \rho_N$.
By the definitions of $\bar \bT, \bar T$ from \eqref{barT-def}, \eqref{T-def}, we have
\begin{align}
\frac{\bar \bT\bar \rho_N}{\bar \rho_N}(\by)&= \sum_{i=1}^N \frac{(\bar T\bar \rho)(x_i,\xi_i)}{\bar \rho(x_i,\xi_i)} \nonumber \\
&= \sum_{\substack{R:(k,l)\to (k',l')}} \sum_{i=1}^N \frac{ \chi_{k'}(\xi_i)(\Phi_R*u_l)(x_i) u_k(x_i)+ \chi_{l'}(\xi_i)(\Phi_R*u_k)(x_i)u_l(x_i) }{\bar \rho(x_i,\xi_i)} \nonumber \\
&\qquad - \sum_{R:(k,l)\to(k',l')} \sum_{i=1}^N [\chi_k(\xi_i)(\Phi_R*u_l)(x_i)+ \chi_l(\xi_i)(\Phi_R*u_k)(x_i)] \nonumber \\
&= \sum_{R\in \fr} \sum_{i=1}^N [A^R_t(y_i)+\hat A^{R}_t(y_i)], \label{sum1}
\end{align}
where, for a given reaction $(k,l)\xrightarrow{R} (k',l')$ from $\fr$,
\begin{align}
A_t^R(y)&:= \frac{\chi_{l'}(\xi)(\Phi_R*u_k)(x)u_l(x) }{\bar \rho(x,\xi)} - \chi_{l}(\xi) (\Phi_R*u_k)(x), \label{A-def}\\
\hat A_t^R(y)&:= \frac{\chi_{k'}(\xi)(\Phi_R*u_l)(x)u_k(x) }{\bar \rho(x,\xi)} - \chi_{k}(\xi)  (\Phi_R*u_l)(x) .\nonumber
\end{align}
We can symmetrize \eqref{sum1} by introducing an extra summation over $j=1,\cdots, N$, so that
\begin{align}\label{comp1}
\frac{\bar \bT(\bar \rho_N)}{\bar \rho_N}(\by)&= \sum_{R\in \fr} \sum_{i=1}^N [A_t^R(y_i)+\hat A_t^R(y_i)]= \f 1 N \sum_{R\in \fr} \sum_{i,j=1}^N \sqb{A_t^R(y_j)+\hat A_t^R(y_i)}.
\end{align}

On the other hand, by the definition of $\bS^*_N$ from \eqref{S-def} and the tensorized law $\bar \rho_N$,
\begin{align}
\frac{\bS^*(\bar \rho_N)}{\bar \rho_N}(\by)
&=\f 1 N \sum_{R\in \fr}\sum_{i,j=1,i\neq j}^N \Phi_R(x_i-x_j)\sqb{\chi_R^+(\xi_i,\xi_j) \frac{\bar \rho(x_i,k) \bar\rho(x_j,l)}{\bar \rho(x_i,\xi_i)\bar \rho(x_j,\xi_j)} -\chi_R^-(\xi_i,\xi_j) }\nonumber \\
&= \f 1 N \sum_{R\in \fr} \sum_{i,j=1,i\neq j}^N B_t^R(y_i,y_j), \label{comp2}
\end{align}
where
\begin{align}
B_t^R(y,y')&:= \Phi_R(x-x')\sqb{\chi_R^+(\xi,\xi')\frac{u_k(x)u_l(x')}{\bar \rho(x,\xi)\bar \rho(x',\xi')} - \chi_R^-(\xi,\xi')} . \label{B-def}
\end{align}

Now define
\begin{align}
F_t^R(y_i,y_j):= A_t^R(y_j)+\hat A_t^R(y_i)-B_t^R(y_i,y_j), \quad
f_t(y_i,y_j):= \sum_{R\in \fr} F_t^R(y_i,y_j). \label{F-def}
\end{align}
By \eqref{comp1}, \eqref{comp2}, we have
\begin{align*}
\frac{(\bar \bT-\bS^*)(\bar \rho_N)}{\bar \rho_N}(\by)
&= \f 1 N \sum_{i,j=1,i\neq j}^N f_t(y_i,y_j)+ \f 1 N \sum_{R\in \fr}\sum_{i=1}^N \sqb{A_t^R(y_i)+\hat A_t^R(y_i)}.
\end{align*}
and so the quantity $G(t)$ is given by
\begin{align*}
G(t)&= \f 1N \int_{\Pi^N} \rho_N \frac{(\bar \bT-\bS^*)(\bar \rho_N)}{\bar \rho_N}dm^N \\
&= \f 1 {N^2} \sum_{i,j=1,i\neq j}^N  \int_{\Pi^N} \rho_Nf_t(y_i,y_j)dm^N + \f 1 {N^2}\sum_{R\in \fr,i=1}^N \int_{\Pi^N} \rho_N\sqb{A_t^R(y_i)+\hat A_t^R(y_i)} dm^N \\
&=:G_1(t)+G_2(t).
\end{align*}

Introduce the quantity
\begin{align}\label{K-def}
K_t&:= \sum_{R\in \fr} \cb{\sup_{y\in \Pi}\sqb{|A_t^R(y)|+|\hat A_t^R(y)|}+\sup_{(y,y')\in \Pi^2} |B_t^R(y,y')|} .
\end{align}
Then the diagonal term $G_2(t)$ is simply bounded by
\begin{align*}
G_2(t)&= \f 1{N^2}\sum_{R\in \fr}\sum_{i=1}^N \int_{\Pi^N}	\rho_N(\by)\sqb{A_t(y_i)+\hat A_t(y_i)}dm^N(\by)
\le \f{K_t}{N}.
\end{align*}
As for $G_1(t)$, to set up a differential inequality for Gr\"onwall lemma, we need the following inequality which essentially follows from the variational characterization for relative entropy.
\begin{lemma}[{\cite[Lemma 1]{JW18}}]
	Let $N\ge 1$ and $\rho,\bar \rho$ be two probability measures on the space $\Pi^N$. For every $\eta>0$ and $\Psi\in L^\infty(\Pi^N)$, it holds
	\begin{align*}
	\int_{\Pi^N} \Psi d\rho&\le \eta\sqb{\mch_N(\rho\|\bar \rho) + \f 1 N \log \int_{\Pi^N} e^{\eta^{-1}N\Psi} d\bar \rho  } .
	\end{align*}
\end{lemma}

Applying this lemma to $G_1(t)$, with $\Psi=N^{-2}\sum_{i,j=1,i\neq j}^N f_t(y_i,y_j)$, $\rho_N(t),\bar\rho_N(t)$ in place of $\rho,\bar \rho$, 
and with some $\eta>0$ (depending on $f_t$, c.f. \eqref{F-def}) to be determined later, we have
\begin{align}\label{G-term}
G_1(t)&\le \eta W_N(t)+\f {\eta} N \log \int_{\Pi^N}\bar \rho_N \exp \rb{\f 1 {\eta N} \sum_{i,j=1}^N f_t(y_i,y_j) } dm^N.
\end{align}
To this end, we need an estimate of the exponential moment on the right hand side.
This can be achieved by using the following large deviation inequality.

\begin{lemma}\label{Concen}
	Let $(\Pi,\bar \rho)$ be a probability space, $\{Y_1,Y_2,\cdots\}$ be a sequence of i.i.d.
	$\Pi$-valued random variables with common distribution $\bar \rho$, and $f\in L^\infty(\Pi^2,\bar \rho^{\otimes 2})$ 
	be a bounded measurable function satisfying the following marginal mean zero conditions:
	\begin{align}\label{eq:center}
	\mbe [f(Y_1,Y_2)|Y_1]=0,\qquad \mbe [f(Y_1,Y_2)|Y_2]=0.
	\end{align}
	Then there exists a constant $\eta=\eta(f)>0$ such that
	\begin{align}\label{exp-mom}
	\sup_{N\ge 2}\mbe \sqb{\exp\rb{\f 1 {\eta N} \sum_{i,j=1,i\neq j}^N f(Y_i,Y_j)}  }\le 2.
	\end{align}
	Specifically, $\eta(f)$ can be chosen to be $2\sqrt{2}e\|f\|_{L^\infty(\Pi^2,\bar \rho^{\otimes 2})}$.
\end{lemma}

As mentioned earlier, Lemma \ref{Concen} was first proved by Jabin and Wang \cite{JW18} (see Theorem 4 therein), 
and we will provide a shorter proof in the coming section, after completing the proof of the main result.
To finish the estimation, we apply this lemma to the second integral of \eqref{G-term}.
Prior to that, we first verify the marginal mean zero condition \eqref{eq:center}, required by the lemma, for the function $f_t(y,y')$.

\begin{lemma}
	The function $f_t(y,y')$ defined in \eqref{F-def} satisfies the marginal mean zero conditions \eqref{eq:center}:
	\begin{align*}
	\int _\Pi f_t(y,y')\bar \rho(t,y)dm(y) =0=	\int _\Pi f_t(y,y')\bar \rho(t,y')dm(y').
	\end{align*}
\end{lemma}

\begin{proof}
	Recall from the definition \eqref{F-def} that $f_t(y,y')=\sum_{R\in \fr}F_t^R(y,y')$, where $F_t^R(y,y')=A_t^R(y')+\hat A_t^R(y)-B^R_t(y,y')$, with $A_t^R,\hat A_t^R, B_t^R$ from \eqref{A-def}, \eqref{B-def}. The lemma is a direct consequence of
	the following four identities:
	\begin{align*}
	&\int_\Pi A_t^R(y)\bar \rho(y)dm(y)=0=\int_\Pi \hat A_t^R(y)\bar \rho(y)dm(y),
	\end{align*}
	\begin{align*}
	\int_\Pi B_t^R(y,y')\bar \rho(y)dm(y)=A_t^R(y'),\qquad \int_\Pi B_t^R(y,y')\bar \rho(y')dm(y')=\hat A_t^R(y).
	\end{align*}
	Let us begin with the first one. By \eqref{A-def},
	\begin{align*}
	\int _\Pi A_t^R(y)\bar \rho(y)dm(y)&= \sum_{\xi=1}^n  \int_\mbx \cb{\chi_{l'}(\xi) (\Phi_R*u_k)(x)u_l(x) - \chi_{l}(\xi) \bar \rho(x,\xi) (\Phi_R*u_k)(x)} dx\\
	&= \int_{\mbx} [(\Phi_R*u_k)(x)u_l(x)-(\Phi_R*u_k)(x)u_l(x)]dx=0.
	\end{align*}
	This establishes the first identity.
	Swapping the role of $k$ and $l$, and replacing $l'$ by $k'$ in the computation above yields the second one.
	Now consider the third identity. Again by the definition of $B_t^R$ (see \eqref{B-def}), we have
	\begin{align*}
	&\qquad \int_{\Pi} B_t^R(y,y')\bar \rho(t,y)dm(y)\\
	&= \int _\Pi \Phi_R(x-x') \sqb{\chi_R^+(\xi,\xi') \frac{u_k(x)u_l(x')}{\bar \rho(y)\bar \rho(y')} -\chi_R^-(\xi,\xi') } \bar \rho(y)dm(y)\\
	&= \frac{u_l(x')}{\bar \rho(y')} \sum_{\xi=1}^n \int _\mbx \chi_R^+(\xi,\xi') \Phi_R(x-x')u_k(x)dx- \sum_{\xi=1}^n\int_\mbx  \chi_R^-(\xi,\xi')\Phi_R(x-x')\bar \rho(x,\xi)dx\\
	&= \frac{u_l(x')}{\bar \rho(y')} \chi_{l'}(\xi') (\Phi_R*u_k)(x')- \chi_{l}(\xi') (\Phi_R*\bar \rho)(x',\xi)=A^R_t(y').
	\end{align*}
	The last step is due to $\sum_{\xi=1}^n \chi_R^+(\xi,\xi')=\chi_{l'}(\xi')$, $\sum_{\xi=1}^n \chi_R^-(\xi,\xi')=\chi_{l}(\xi')$ (see \eqref{chi-def}).
	Again, swapping the role of $k,l$ gives the fourth identity.

From these calculations, it follows that
\begin{align}
\int_\Pi f_t(y,y')\bar \rho(t,y)dm(y) & = \sum_{R\in \fr} \int_\Pi \left( A_t^R(y')+\hat A_t^R(y)-B^R_t(y,y')\right) \bar \rho(t,y)dm(y) \nonumber \\ 
 & =  \sum_{R\in \fr} A_t^R(y') + 0 - A_t^R(y') = 0, \nonumber
\end{align}
and similarly for $\int_\Pi f_t(y,y')\bar \rho(t,y')dm(y') = 0$.
\end{proof}

Applying Lemma \ref{Concen} to the integral from \eqref{G-term} with $\bar \rho(t),f_t$ in place of $\bar \rho,f$, and recall the quantity \eqref{K-def}, it follows
\begin{align*}
G_1(t)&\le \eta W_N(t)+ \f {\eta\log 2} N ,\qquad  \eta := 2\sqrt 2e\|f_t\|_{L^\infty(\Pi^2,\bar \rho^{\otimes 2})}\le 2\sqrt 2 eK_t.
\end{align*}

\subsection{Conclusion}
Substituting $D \leq 0$ and the estimates for $G_1,G_2$ into \eqref{RE-ineq}, we have
\begin{align}\label{diff-eq}
W_N'(t)&\le CK_t \rb{W_N(t)+\f {1} N},
\end{align}
for some universal constant $C>0$ and $K_t$ from \eqref{K-def}.
Let us now get a bound for $K_t$, assuming $\bar \rho$ satisfies the comparability condition \eqref{comparability}.
Note that if $\mbx=\mathbb{T}^d$, then the solution $\bar \rho(t)$ of the system \eqref{MFE} is locally bounded and strictly positive,
because of the assumption of Theorem \ref{main} $\bar \rho_0\in L^\infty(\Pi), \inf_{y\in \Pi}\bar\rho(0,y)>0$  (see Propositions \ref{well-posed}, \ref{positive}).
Therefore, for every $T>0$, the condition \eqref{comparability} holds for some constant $C_T>0$ depending on $\bar \rho_0$.
By \eqref{A-def}, \eqref{B-def}, and Young's inequality, it follows
\begin{align*}
\sup_{y\in \Pi,t\in [0,T]} \cb{|A_t^R(y)|,|\hat A_t^R(y)|}&\le \|\Phi_R\|_{L^\infty}\rb{C_T+1} ,\\
\sup_{(y,y')\in \Pi^2,t\in [0,T]} |B_t^R(y,y')|&\le \|\Phi_R\|_{L^\infty}\rb{C_T^2+1}.
\end{align*}
Hence, we have
\begin{align*}
\sup_{t\in [0,T]} K_t &\le\rb{C_T^2+2C_T+3}  \sum_{R\in \fr} \|\Phi_R\|_{L^\infty}=:\til C_T\|\Phi\|_{L^\infty},
\end{align*}
where recall that $\|\Phi\|_{L^\infty}=\sum_{R\in\fr}\|\Phi_R\|_{L^\infty}<\infty$, which is bounded
(as $\fr$ is a finite set).
Now applying Gr\"onwall lemma to \eqref{diff-eq}, we establish the bound
\begin{align*}
W_N(t) \le e^{C t \til C_T\|\Phi\|_{L^\infty}}  W_N(0)+ \frac{e^{C t \til C_T\|\Phi\|_{L^\infty}} - 1}{N}, \quad \quad \forall \;\;t \in [0,T].
\end{align*}
The desired estimate from \eqref{RE-bound} follows, if we set $\Lambda_T=C \til C_T$. 
This finishes the proof of the main theorem.


\section{The Large Deviation Inequality (Proof of Lemma \ref{Concen})} \label{LDP-ineq}
The main objective of this section is to prove the large deviation inequality \eqref{exp-mom} from Lemma \ref{Concen}.
Before proceeding to its proof, let us first make a few comments about the result.
In Lemma \ref{Concen} we do not require that $f$ is continuous.
In fact, if $f$ is bounded and continuous on $\Pi^2$, then the classical large deviation principle of empirical measures (see e.g. \cite{BB90}) would imply that
\begin{align*}
\limsup_{N\rightarrow \infty} e^{N m}\cdot  \mbe \sqb{\exp\rb{\f 1 {\eta N} \sum_{i,j=1,i\neq j}^N f(Y_i,Y_j)}  } < \infty,
\end{align*}
where the constant $m$ is characterized by
\begin{align*}
m & = \inf_{\mu \in \mathcal{P}(\Pi)} \Big\{\mch (\mu \| \bar{\rho}) - \iint \frac{f(x, y)}{\eta} \mu(dx) d\mu(dy)\Big\}
\\
& =  \inf_{\mu  \in \mathcal{P}(\Pi)} \Big\{\mch (\mu \| \bar{\rho}) - \iint \frac{f(x, y)}{\eta} (d\mu(x) -d\bar{\rho}(x))  (d\mu(y) -d\bar{\rho}(y)) \Big\}.
\end{align*}
The last line above follows from the mean zero condition \eqref{eq:center}. From above one obtains that $m = 0$ if $\eta $ is chosen large enough, which proves
\begin{align*}
\lim_{N\to\infty} \f 1 N \log \mbe \sqb{\exp\rb{\f 1 {\eta N} \sum_{i,j=1,i\neq j}^N f(Y_i,Y_j)}  }= 0.
\end{align*}
Despite being a weaker estimate compared to \eqref{exp-mom}, this qualitative result alone is sufficient to 
conclude the propagation of chaos of the particle system (at least for the case of reaction kernels $\Phi_R$ being continuous), except without an explicit bound. The quantitative estimate \eqref{exp-mom} was proved by Jabin and Wang \cite{JW18} under the assumption that $f$ satisfies
$$
C\sup_{p\geq 1} \Big(\frac{\|\sup_z |f(\cdot, z)|\|_{L^p(\bar{\rho}dx)}}{p}\Big)^2 <1 
$$
for some constant $C>0$, which is a weaker assumption than $\|f\|_{L^\infty} <\infty$. Their proof relies on some sophisticated
combinatorial analysis by taking  account of cancellations due to the mean zero condition \eqref{eq:center}.
We provide a simple probabilistic alternative.
Our proof to Lemma \ref{Concen} relies on a characterization of the exponential moment, given as follows:
\newcommand{\mcf}{\mathcal {F}}

\begin{lemma}\label{exp-lem}
	Let $Z$ be a random variable satisfying the bound for some $\ga>0$:
	\begin{align*}
	\sup_{k\in \mbn} k^{-1} |\mbe Z^k|^{1/k}\le \ga.
	\end{align*}
	Then it holds
	\begin{align*}
	\mbe[e^{(2e\ga)^{-1}Z}]\le 2.
	\end{align*}
\end{lemma}

\begin{proof}
	Let $\eta=2 e \ga$. Expanding $e^{\eta^{-1}Z}$ by Taylor series, and using the assumption,
	\begin{align*}
	\mbe \sqb{e^{\eta^{-1}Z} } &=\sum_{k=0}^\infty \frac{1}{\eta ^k k!} \mbe\sqb{Z^k}
	\le \sum_{k=0}^\infty \frac{\ga^k k^k}{\eta^kk!}\le \sum_{k=0}^\infty \rb{\frac{e\gamma}{\eta}}^k=\sum_{k=0}^\infty 2^{-k}=2.
	\end{align*}
	The only inequality above follows by the simple inequality $k^k(k!)^{-1}\le e^k$, which follows from the Stirling's approximation.
\end{proof}

In the proof of \eqref{exp-mom} we use \emph{a Marcinkiewicz-Zygmund type inequality} for martingales, which, roughly speaking, 
bounds the $L^p$-norm of a martingale by the root sums squared of the $L^p$-norm of its martingale increments.
Specifically, we will use the following sharp version of inequality \eqref{eq:martingale}, due to Rio \cite{Rio09}.
When $p=2$, \eqref{eq:martingale} holds as an equality, which is due to the It\^{o} isometry for discrete martingales.
Also, the constant $p-1$ on the right hand side of \eqref{eq:martingale} is known to be sharp.
We point out also the sharpness of this estimate, specifically the growth rate as $p\to\infty$, plays a critical role in our argument.

\begin{lemma}[{\cite[Theorem 2.1]{Rio09}}]
\label{eq:MZ}
Let $p \geq 2$, and $\{X_k\}_{k\ge 1}$ be a sequence of $L^p$-martingale differences with respect to a filtration $\{\mcf_k\}_{k\ge 0}$. 
That is, for each $k\ge 1$, $X_k$ is $\mcf_k$-measurable, $X_k\in L^p$, and $\mbe [X_k|\mcf_{k-1}]=0$.
Let $S_n=\sum_{k=1}^nX_k$ (which is a martingale). Then for all $n\ge 1$ we have
\begin{equation}\label{eq:martingale}
\|S_n\|_{L^p}^2\leq (p-1) \sum_{k=1}^n\|X_k\|_{L^p}^2.
\end{equation}
\end{lemma}

Now we are ready to prove Lemma \ref{Concen}.
\begin{proof}[Proof of Lemma \ref{Concen}]
	Let $\{Y_j\}_{j\in \mbn}$ be a sequence of $\Pi$-valued i.i.d. random variables,
	and $f\in L^\infty(\Pi^2,\bar \rho^{\otimes 2})$ as in the lemma.
	For $N\ge 2$, denote the random variables
	\begin{align*}
	A_N(f)& := \f 1 N \sum_{i,j=1, i\neq j}^N f(Y_i,Y_j),\qquad M_N:=NA_N(f)=\sum_{i,j=1,  i\neq j}^N f(Y_i,Y_j).
	\end{align*}
	By Lemma \ref{exp-lem}, the statement \eqref{exp-mom} follows if we show the following uniform-in-$N$ bound is valid:
	\begin{align}\label{Lk-bound}
	\sup_{k\in \mbn} k^{-1}|\mbe A_N(f)^k|^{1/k}\le \sqrt{2}\|f\|_{L^\infty(\Pi;\bar \rho^{\otimes 2})}.
	\end{align}

	We first write $M_N$ as the sum of martingale differences.
	Indeed, we have
	\begin{align*}
	M_N&=\sum_{k=1}^N D_k,  \quad\mbox{where } D_k= \sum_{i=1}^{k-1}f(Y_i,Y_k)+\sum_{j=1}^{k-1}f(Y_k,Y_j),
	\end{align*}
	($D_1$ is set to be zero).
	Observe here that $\{D_k\}_{k=1,\cdots, N}$ forms a $L^p$ martingale difference w.r.t. the 
	filtration $\{\mcf_k=\si(Y_1,\cdots, Y_k)\}_{k\ge 0}$ ($\mcf_0$ is set to be the trivial $\si$-algebra) for any $p\in [1,\infty)$ (and hence $\{M_k\}_{1\le k\le N}$ is an $L^p$-martingale).
	Indeed, each $D_k\in L^p$ because $f$ is bounded. Moreover, the marginal mean zero condition \eqref{eq:center} implies
	\begin{align*}
	\mbe [D_k|\mcf_{k-1}]&= \sum_{i=1}^{k-1}\mbe[f(Y_i,Y_k)+f(Y_k,Y_i)|Y_1,\cdots, Y_{k-1}]=0.
	\end{align*}
	Using Lemma \ref{eq:MZ} (with $N,M_N,D_k$ in place of $n,S_n,X_k$), we establish the following bound for all $p\ge 2$:
	\begin{align}\label{M-ineq}
	\|M_N\|_{L^p}\le \sqrt{p-1}\rb{\sum_{k=1}^N \|D_k\|_{L^p}^2 }^{1/2}.
	\end{align}

	Now consider $\|D_k\|_{L^p}$ for each $1\le k\le N$.
	We again write $D_k$ as a sum of martingale difference, namely,
	\begin{align*}
	D_k=\sum_{j=1}^{k-1} B^k_j,\qquad \mbox{where } B^{k}_j=f(Y_k,Y_j)+f(Y_j,Y_k)\quad  \mbox{ for }1\le j\le k-1,
	\end{align*}
	Note that $\|B^{k}_j\|_{L^p}\le 2\|f\|_{L^\infty(\Pi^2,\bar \rho^{\otimes 2})}$ for each $k,j$, and $\{B_j^k\}_{1\le j\le k-1}$ 
	again forms a sequence of martingale differences w.r.t. the filtration $\{\til \mcf_j\}_{0\le j\le k-1}$, 
	where $\mcf_j=\si(Y_k,Y_1,\cdots, Y_j)$. Specifically, $\mbe [B_j^k|\til \mcf_{j-1} ]=0$ follows directly from the marginal mean zero condition \eqref{eq:center}.
	So we may again apply Lemma \ref{eq:MZ} (with $k-1,D_k,B^{k}_j$ in place of $n,S_n,X_k$) to obtain
	\begin{align*}
	\|D_k\|_{L^p}&\le \sqrt{p-1} \rb{ \sum_{j=1}^{k-1} \|B_j^k\|_{L^p}^2 }^{1/2} \le 2 \sqrt {(k-1)(p-1)} \|f\|_{L^\infty(\Pi^2)}.
	\end{align*}
	Inserting this estimate into \eqref{M-ineq}, it follows
	\begin{align*}
	\| M_N\|_{L^p}&\le 2(p-1)\|f\|_{L^\infty(\Pi^2)} \rb{\sum_{k=1}^N (k-1)}^{1/2}\le \sqrt{2} (p-1)N\|f\|_{L^\infty(\Pi^2)} .
	\end{align*}

	Now we return to $A_N(f)=N^{-1} M_N$. Using the $L^p$-bound for $ M_N$ that we just established, we have for all $k\ge 2$ that
	\begin{align*}
	\f 1 k \rb{|\mbe A_N(f)|^k}^{1/k}&\le \f 1 {Nk} |\mbe M_N^k|^{1/k}
	\le \f 1 {Nk}\| M_N\|_{L^k}
	\le \sqrt 2 \|f\|_{L^\infty(\Pi^2)}.
	\end{align*}
	For $k=1$, we have $|\mbe A_N(f)|= 0$ by \eqref{eq:center}. This concludes the proof of \eqref{Lk-bound}.
\end{proof}


\section{Appendix: Well-posedness and Regularity of Semilinear Parabolic Systems}\label{appen}
In this appendix, we provide a brief discussion on the systems \eqref{FP-eq} and \eqref{MFE}, particularly,  the well-posedness of the correspondent Cauchy problem, and regularity of solutions.
These results follow by the classical theory of semigroups and the standard construction of solutions to ODEs using the contraction mapping theorem.
Though elementary, for the sake of completeness we also present their proofs.
For a detailed discussion, we point the reader to, for instance, \cite{Pazy}.

To state a result that is applicable for both the (linear) Fokker-Planck equation \eqref{FP-eq} on $\Pi^N$ and the (nonlinear) mean field limit system \eqref{MFE} on $\Pi$, we will consider a general form of equation. Let $\mbx=\mathbb{T}^D$ or $\mbr^D$, with spatial dimension $D\in \mbn$, $G$ be a finite set of indices, and denote $\Ga=\mbx\times G$. Denote also the variables $y=(x,\xi)\in \mbx\times G$, and the measure $dm=dx\otimes d\#$ on $\mbx\times G$, where $\#$ denotes the counting measure on the index set $G$. 
For $p\in [1,\infty]$, denote the Banach space and norm
\begin{align*}
X_p=L^1(\Ga,dm)\cap L^p(\Ga,dm),\qquad \|\cdot\|_{X_p}=\| \cdot \|_{L^1(\Ga)}+\|\cdot\|_{L^p(\Ga)},
\end{align*}
and $X_p^+\subset X_p$ be the (closed) subset of all nonnegative $(L^1\cap L^p)(\Ga)$-functions. 
Similarly, for $k\ge 0$ and $p\in [1,\infty]$, we denote $W^{k,p}(\Ga)$ the Banach space of all functions $f$ on $\Ga$ so that $f(\cdot,\xi)\in W^{k,p}(\mbx)$ for every $\xi\in G$, with the norm $\|f\|_{W^{k,p}(\Ga)}= \sum_{\xi\in G} \|f(\cdot,\xi)\|_{W^{2,p}(\mbx)}$.

We consider the evolution equation on $\mbx\times G$, given by
\begin{align}\label{ODE2}
\partial_t \rho = A\rho + T^-(\rho)+T^+(\rho),\quad (t,x,\xi)\in (0,\infty)\times \mbx\times G,\qquad \rho(0)=\rho_0\in X_p^+,
\end{align}
where $A$ is the $D$-dimensional elliptic operator on $\mbx\times G$ given by
\begin{align}
A\rho(x,\xi)&= \sum_{k=1}^D \al_k(\xi) \partial_{k}^2\rho(x,\xi), \qquad \al_k(\xi)>0\mbox{ for all }1\le k\le D,\xi\in G,\label{ellip}
\end{align}
and $T^\pm:L^1_{loc}(\Ga)\to L^1_{loc}(\Ga)$ (possibly nonlinear).
If regarding $\xi\in G$ as an index, one may view the forward equation \eqref{ODE2} as a parabolic system with $n=|G|$ equations in variables $(t,x)\in (0,\infty)\times \mbx$. 
Namely, for each $\xi\in G$, the component $u_\xi(t,x) =\rho(t,x,\xi)$ satisfies the parabolic equation
\begin{align}\label{parabolic}
\partial _t u_\xi = A_\xi u_\xi + T^-(\rho)(t,x,\xi) +T^+(\rho)(t,x,\xi),\quad (t,x)\in (0,\infty)\times \mbx,
\end{align}
where $A_\xi$ is the elliptic operator on $\mbx$ given by \eqref{ellip} with $\xi\in G$ fixed.

Throughout this section we assume the following hypothesis for the maps $T^\pm$: it holds for some constant $C>0$ that
\begin{enumerate}[label=(T\arabic*)]
	\item (\textbf{Boundedness and local Lipschitz continuity}). 
	For every $p\in [1,\infty)$ and $\rho,\hat \rho\in X_p$, it holds 
	\begin{align}
	\|T^\pm(\rho)\|_{X_p}&\le C\max\{1,\|\rho\|_{L^1(\Ga )}\} \| \rho\|_{X_p}, \nonumber \\
	\|T^\pm(\rho)-T^\pm(\hat \rho)\|_{X_p}&\le C\max\{1,\|\rho \|_{X_p},\|\hat \rho\|_{X_p}\} \|\rho-\hat \rho\|_{X_p}; \label{Lip-bound} 
	\end{align}
	\item (\textbf{Nonnegativity of $T^+$}). $T^+$ maps nonnegative functions to nonnegative functions;
	\item (\textbf{Pointwise bound of $T^-$}). For every $\rho\in L^1(\Ga)$, we have
	\begin{align*}
	|T^-(\rho)(x,\xi)|\le C\max\{1,\|\rho\|_{L^1(\Ga)}\} |\rho (x,\xi)|,\quad \mbox{ for $m$-a.e. }(x,\xi)\in \Ga;
	\end{align*} 
	\item (\textbf{Mass conservation}). For every $\rho\in L^1(\Ga)$, it holds
	\begin{align*}
	\int_{\Ga} [T^-(\rho) +T^+(\rho)]dm =0. 
	\end{align*}
\end{enumerate}

Both forward equations \eqref{FP-eq} and \eqref{MFE} are special cases of \eqref{ODE2}. For the Fokker-Planck equation \eqref{FP-eq} on $\Pi^N$, we have the spatial domain $\mbx=\mathbb{T}^{dN}$ or $\mbr^{dN}$, the index set $G=\fs^N$ (and so $D=dN$, $\Ga=\Pi^N$),
$A=\bDe_N$ from \eqref{Ga-def}, and the maps $T^\pm$, by \eqref{S-def}, are given by
\begin{align*}
T^-(\rho)(\by)&= -\f 1 N \sum_{R\in \fr} \sum_{\substack{i,j=1, i\neq j}}^N  \Phi_R(x_i-x_j)\chi_R^-(\xi_i,\xi_j)\psi(\by) ,\qquad T^+= \bS_N^*-T^-.
\end{align*}
Since in this case $T^\pm$ are bounded linear in $X_p$ (as $\Phi_R\in L^\infty$), (T1)--(T4) hold with $C>0$ depending on $N$ and $\|\Phi\|_{L^\infty}=\sum_{R\in \fr}\|\Phi_R\|_{L^\infty}$. 
((T4) is a straightforward computation from the definition \eqref{S-def} of $\bS_N^*$.)
For the limit system \eqref{MFE}, we have $\mbx = \mathbb{T}^d$ or $\mbr^d$, $G=\fs$ (hence $D=d$, $\Ga=\Pi$), and $A=\f 12 \si(\xi)^2 \De$. The maps $T^\pm$ are given by $T^\pm = \sum _{R\in\fr} \bar T_R^\pm$ from \eqref{T-def}.
Conditions from (T2)--(T4) can be easily verified. 
For (T1), the boundedness condition (with $C$ depending on $\|\Phi\|_{L^\infty}$) is a direct consequence of Young's inequality and the definition of $\bar T_R^\pm$ from \eqref{TR-def}. 
To show the local Lipschitz bound \eqref{Lip-bound}, it suffices to show \eqref{Lip-bound} for a fixed $\bar T_R^\pm$, $R\in \fr$.
For any $\bar \rho,\hat \rho\in X_p$, by \eqref{TR-def} it holds for all $(x,\xi)\in \Pi$ that
\begin{align*}
|[\bar T_R^\pm(\bar \rho)-\bar T_R^\pm(\hat \rho)](x,\xi)|&\le 2|(\bar \rho-\hat \rho)(x,\xi)| \|\Phi_R*\bar \rho\|_{L^\infty(\Pi)}+2 |\hat \rho(x,\xi)| \|\Phi_R*(\bar \rho-\hat \rho)\|_{L^\infty(\Pi)}.
\end{align*}
Taking $\|\cdot\|_{X_p}$-norm for the above, and applying Young's inequality, we have
\begin{align*}
\|\bar T_R^\pm(\bar \rho)-\bar T_R^\pm (\hat \rho)\|_{X_p}&\le 2\|\bar \rho-\hat \rho\|_{X_p} \|(\Phi_R*\bar \rho)(\cdot ,\xi)\|_{L^\infty(\Pi)} +2 \|\hat \rho\|_{X_p} \|\Phi_R*(\bar \rho-\hat \rho)\|_{L^\infty(\Pi)}\\
&\le 2\|\Phi_R\|_{L^\infty} \sqb{\|\bar \rho\|_{L^1(\Pi)}\|\bar \rho-\hat \rho\|_{X_p}+ \|\hat \rho\|_{X_p} \|\bar \rho-\hat \rho\|_{L^1(\Pi) } }\\
&\le 4 \|\Phi_R\|_{L^\infty}\max\{\|\bar \rho\|_{X_p},\|\hat \rho\|_{X_p}\}\|\bar \rho-\hat \rho\|_{X_p}.
\end{align*}
Hence, in this case $T^\pm$ also satisfy (T1)--(T4).

Notice that, for any $p\in [1,\infty)$, $A$ is the generator of a contractive (analytic) semigroup on $X_p=(L^1\cap L^p)(\Ga,dm)$, and 
denote $\{e^{tA}\}_{t\ge 0}$ the correspondent semigroup.
This semigroup is in fact \emph{mass-conserving}, in the sense that
\begin{align}\label{mass-consv}
\int_\Ga e^{tA}fdm=\int_\Ga fdm,\quad \forall f\in L^1(\Ga),\quad t\ge 0.
\end{align}
Moreover, by the estimate of heat potentials, the semigroup has the following $L^p\to L^r$ bound for any $1\le p\le r\le \infty$:
\begin{align}\label{LpLr-bound}
\|e^{tA}f\|_{L^r(\Ga)}&\le C_{D,p,r}t^{-\f{D}{2}(\f 1 p - \f 1 r)}\|f\|_{L^p(\Ga)}.
\end{align}
We say a $X_p$-valued function $\rho\in C([0,t_0);X_p)$ for some $t_0\in (0,\infty]$ is a \emph{mild solution} to the equation \eqref{ODE2} if
\begin{align}\label{mild}
\rho(t)&= e^{tA}\rho_0+\int _0^t e^{(t-s)A}[T^-(\rho)+T^+(\rho)](s)ds,\qquad \forall t\in [0,t_0).
\end{align}
It is a \emph{local solution} if $t_0<\infty$, and a \emph{global solution} if $t_0=\infty$.
Notice also if a mild solution possesses higher regularity, namely, $\rho\in C((0,t_0);W^{2,p}(\Ga))\cap C^1((0,t_0);X_p)$, then it is a \emph{strong solution}.
If it is $C^1$ in time, $C^2$ in space (with respect to $x\in \mbx$), then it is a \emph{classical solution}.

The main result of this section, which applies to both Fokker Planck equation \eqref{FP-eq} and mean field limit equation \eqref{MFE}, is stated as follows.
\begin{proposition}\label{well-posed}
	Let $A$ be from \eqref{ellip} and $T^\pm$ satisfy (T1)--(T4).
	For every nonnegative $L^1(\Ga)$ initial data $\rho_0$, the Cauchy problem \eqref{ODE2} admits a unique global mild solution. 
	The solution is nonnegative, and mass-conserving, in the sense that
	\begin{align}\label{mass-conserved}
	\rho(t)\ge 0,\qquad \int_{\Ga} \rho(t)dm=\int_{\Ga}\rho_0dm,\qquad \forall t\in [0,\infty).
	\end{align}
	The solution has the following regularity for any $p\in [1,\infty)$:
	\begin{align*}
	\rho\in C([0,\infty);L^1(\Ga))\cap C((0,\infty);W^{2,p}(\Ga))\cap C^1((0,\infty);L^p(\Ga)).
	\end{align*}
	If additionally, for some $p\in [1,\infty)$ $\rho_0\in L^p(\Ga)$, then $\rho\in C([0,\infty);L^p(\Ga))$; if $\rho_0\in W^{2,p}(\Ga)$, then $\rho\in C([0,\infty);W^{2,p}(\Ga))\cap C^1([0,\infty);L^p(\Ga))$.
\end{proposition}

\begin{proof}
	Let us first address the local well-posedness of the problem \eqref{ODE2} with general $X_p$ initial data with any $p\in [1,\infty)$.
	Since the maps $T^\pm$ are assumed to be locally Lipschitz in $X_p$, by Picard-Lindel\"of theorem for Banach-valued ODEs, a unique local mild solution $\rho \in C([0,t_0);X_p)$ exists for every $\rho_0\in X_p$. 
	
	We will first prove existence, uniqueness, non-negativity and mass conservation of a global mild solution $\rho\in C([0,\infty);X_p^+)$ for every initial data $\rho_0\in X_p^+$, $p\in [1,\infty)$. Proposition \ref{well-posed} is then the special case $p=1$. 
	Given $\rho_0\in X_p^+$, let $\rho\in C([0,t_0);X_p)$ be the mild solution of \eqref{ODE2}, guaranteed by the local well-posedness result earlier, and we now show $\rho$ is nonnegative. 
	To this end, we consider \eqref{ODE2} with $T^+(\rho)$ replaced by any $f\in C([0,\infty);X_p^+)$, that is, 
	\begin{align}\label{ODE}
	\partial _t \mu =A\mu + T^-(\mu)+ f(t,x,\xi),\qquad  \mu(0)=\rho_0\in X_p^+.
	\end{align}
	We will first show that the above admits a unique \emph{nonnegative} local mild solution.
	Existence and uniqueness of a local solution $\mu\in C([0,t_0);X_p)$ follows by the local Lipschitz bound \eqref{Lip-bound} for $T^-$, and the same argument earlier.
	To prove its non-negativity, observe that for each $\xi\in G$, by (T3), \eqref{parabolic} with $f$ in place of $T^+\rho$, and $f\ge 0$, $u_\xi(t,x)=\rho(t,x,\xi)$ is a super-solution of the following semilinear equation:
	\begin{align}\label{super-sol}
	\partial_t u_\xi \ge A_\xi u_\xi -\be(t) |u_\xi|,\quad (t,x)\in(0,t_0)\times \mbx ,\quad \be(t)=C\max\{1,\|\mu(t)\|_{L^1(\Ga)}\}<\infty.
	\end{align}
	$\be(t)$ is finite on $[0,t_0)$ because $\mu\in C([0,t_0);L^1(\Ga))$.
	By the comparison principle for semilinear equations ($v\equiv 0$ is a solution to above), we have $u_\xi(t)\ge 0$, i.e., $\mu(t)\in X_p^+$, for all $t\in [0,t_0)$.
	In conclusion, we have shown that the local solution $\mu$ of \eqref{ODE} is in the space $C([0,t_0);X^+_p)$, if the initial data $\rho_0\in X_p^+$.

	Now let $Y_{t_0}=C([0,t_0];X_p^+)$ for $t_0>0$ small, and $\mathcal{P}(\cdot ;\rho_0):Y_{t_0}\to Y_{t_0}$ be the map sending $f\in Y_{t_0}$ to the solution $\mu \in Y_{t_0}$ of \eqref{ODE}.
	Subsequently, define $\mathcal {Q}(g;\rho_0)=\mathcal{P}(T^+(g);\rho_0)$, which again is a map from $Y_{t_0}\to Y_{t_0}$, because $T^+(g)\in X_p^+$ by (T2).
	Again, by choosing $t_0>0$ sufficiently small, using the Lipschitz bound from (T1) for $T^+$, one can show that $\mathcal{Q}(\cdot;\rho_0)$ is contractive in $Y_{t_0}$, and hence has a unique fixed point $\mu\in Y_{t_0}$.
	Since the fixed point also satisfies the equation \eqref{ODE2}, by uniqueness $\mu$ agrees with the unique (local) mild solution $\rho$ of \eqref{ODE2}. Thus, $\rho \in C([0,t_0];X_p^+)$.
	From here, the conservation of mass \eqref{mass-conserved} holds for all $t\in [0,t_0]$, as a direct consequence of \eqref{mass-consv}, \eqref{mild} and (T4).

	To show that the solution can be extended globally, it suffices to establish a global bound for the $\|\cdot\|_{X_p}$-norm of solutions.
	By the bound from (T1), \eqref{mild}, contractive nature of the semigroup $\{e^{tA}\}_{t\ge 0}$ and conservation of mass, a mild solution $\rho(t)$ to \eqref{ODE2} satisfies
	\begin{align*}
	\|\rho(t)\|_{X_p}&\le \|\rho_0\|_{X_p}+\int _0^t \sqb{\|T^-\rho(s)\|_{X_p}+ \|T^+\rho(s)\|_{X_p}} ds \\
	&\le \|\rho_0\|_{X_p}+2 C \int_0^t \max\{1,\|\rho(s) \|_{L^1(\Ga)}\} \|\rho(s)\|_{X_p} ds
	= \|\rho_0\|_{X_p}+C'  \int_0^t \ \|\rho(s)\|_{X_p} ds,
	\end{align*}
	where $C'=2C\max\{1,\|\rho_0\|_{L^1(\Ga)}\}$.
	Applying Gr\"onwall lemma, we then have $\|\rho(t)\|_{X_p}\le \| \rho_0\|_{X_p} e^{C' t}$.
	With this bound, the solution can be extended indefinitely in time, and thus a unique global mild solution exists. 
	To conclude, we have just shown if $\rho_0\in X_p^+$ for some $p\in[1,\infty)$, then \eqref{ODE2} admits a unique global mild solution $\rho\in C([0,\infty);X_p^+)$.
	
	Now we address the regularity issue. 
	From now on let us assume $\rho_0\in X_1^+$, and let $\rho\in C([0,\infty);X_1^+)$ be the unique global mild solution guaranteed by the previous existence and uniqueness result. 
	We now prove that $C((0,\infty);X_p^+)$ for any $p\in[1,\infty)$. 
	Using an induction argument, this follows by the following claim: if $\rho_N\in C((0,\infty);X_p^+)$ for some $p\ge 1$, and $r>p$ is such that $p^{-1}-r^{-1}=D^{-1}$ (recall $D$ is the spatial dimension of $\mbx$), then $\rho_N\in C((0,\infty);X_r^+)$.
	Indeed, regarding $\rho_N$ as a solution of \eqref{ODE2} starting at $t=\de_0$ for some fixed $\de_0>0$,  by (T1), \eqref{mild} and the $L^p\to L^r$ bound \eqref{LpLr-bound}, we have for all $t>0$ that
	\begin{align*}
	\|\rho_N(t+\de_0)\|_{L^r(\Ga)}&\le Ct^{-\f 12 }\|\rho_N(\de_0)\|_{L^p(\Ga)}+ C\int_0^t  (t-s)^{-\f 12 }\|\rho_N(s+\de_0)\|_{X_p} ds\\
	&\le Ct^{-\f 12 }\|\rho_N(\de_0)\|_{X_p}+ Ct^{1/2}\max_{\de_0\le s \le \de_0+t} \|\rho_N(s)\|_{X_p}<\infty,
	\end{align*}
	with some constant $C=C(D,p,r, \|\rho_0\|_{L^1(\Ga)})$.
	Now regarding $\rho_N$ as a solution of the forward equation \eqref{ODE2} starting at $t=2\de_0$, with initial data $\rho_N(2\de_0)\in X_r^+$, the previous existence and uniqueness result then implies $\rho_N\in C([2\de_0,\infty);X_r^+)$. 
	Since $\de_0>0$ is arbitrary, it follows $\rho_N\in C((0,\infty);X_r^+)$, which is our claim.

	We have shown $\rho_N\in C((0,\infty);L^p(\Ga))$ for any $p\in [1,\infty)$.
	By (T1) we have $T^\pm(\rho)\in C((0,\infty);L^p(\Ga))$. 
	By \eqref{parabolic}, since the forcing $T^\pm(\rho)$ is in $L^p(\mbx)$, 
	the standard parabolic regularity results implies $u_\xi(t,x)=\rho(t,x,\xi)\in C((0,\infty);W^{2,p}(\mbx))\cap C^1((0,\infty);L^p(\mbx))$ for every $\xi\in G$, e.g. see Theorem 7.22 of \cite{Lieb96} or Section IV.3 of \cite{LSU68}.
 	Hence, it follows $\rho_N \in C((0,\infty);W^{2,p}(\Ga))\cap C^1((0,\infty);L^p(\Ga))$, for any $p\ge 1$.
	Of course, if $\rho_N(0)\in W^{2,p}(\Ga)$ initially, we may replace the open time interval $(0,\infty)$ above by the closed one $[0,\infty)$.
	This finishes the proof.

\end{proof}

Finally, we give a proof to the strict positivity of solutions correspondent to positive initial data \eqref{ODE2}, for the case of torus $\mbx=\mathbb{T}^D$.

\begin{proposition}\label{positive2}
	Let $\mbx=\mathbb{T}^D$, $\rho_0\in L^1(\Ga)$ with $\rho_0(x,\xi) \geq 0$, and let $\rho$ be the corresponding solution of \eqref{ODE2}.  For any $\xi \in \{ 1,\dots,n\}$, if $\rho_0(\cdot,\xi)$ is positive on a set of positive measure, then $\inf_{(t,x)\in [t_0,t_1]\times \Pi} \rho(t,x,\xi)>0$ for any $0 < t_0 < t_1$. Moreover, if $\inf_{(x,\xi)\in \Ga}\rho_0(x,\xi)>0$, then $\inf_{(t,x,\xi)\in [0,t_1]\times \Pi} \rho(t,x,\xi)>0$ for any $t_1>0$.
\end{proposition}

\begin{proof}
	As shown in the previous proof, for each $\xi\in G$, $u_\xi(t,x)=\rho(t,x,\xi)$ is a super-solution of \eqref{super-sol}. By the conservation of mass, $\be(t)=\be_0$ for some $\be_0>0$ depending on $\|\rho_0\|_{L^1(\Ga)}$. Hence, we have
	\begin{align}
	\partial_t u_\xi\ge A_\xi u_\xi-\be_0 u_\xi,\qquad (t,x)\in (0,\infty)\times \mathbb{T}^D. \label{ulower1}
	\end{align}
	The comparison principle implies
	\begin{align*}
	u_\xi(t,x)&\ge e^{-\be_0 t} H_\xi(t,x)*\rho_0(\cdot,\xi)
	\end{align*}
where $H_\xi$ is the heat kernel on $\mbx = \mathbb{T}^D$ for the uniformly elliptic operator $A_\xi$.  If $\rho_0(\cdot,\xi)$ is positive on a set of positive measure, then $H_\xi(t,x)*\rho_0(\cdot,\xi)$ is strictly positive on any compact subset of $(0,\infty) \times \mathbb{T}^D$.   Moreover, if $\rho_0(\cdot,\xi)$ is bounded below by $\epsilon > 0$, then $H_\xi(t,x)*\rho_0 \geq \epsilon$ also holds.
\end{proof}

Now we give an improved positivity result using the particular structure of the forward equations \eqref{FP-eq} and \eqref{MFE}. Given a subset of species $V_0 \subset \fs$, we inductively define an increasing family of sets $V_n \subset \fs$ by
\begin{align}
V_{n+1} = \bigcup \{ k',\ell' \in \fs \;|;\  \; \{k',\ell'\} = R^+ \;\; \text{for some $R \in \fr$ with $R^- \subset V_n$} \}, \quad n \geq 0. \label{Vndef}
\end{align}
(Recall the notation \eqref{in-output-red} for $R^\pm$.) That is, $V_{n+1}$ is the set of all chemical species (types) that are the products of reactions having inputs only from $V_n$.  Then, we define the set
\[
\overline{V_0} = \bigcup_{n \geq 0} V_n
\] 
We call this set $\overline{V_0} \subset \fs$ the closure of $V_0$ under the reaction network dynamics.  Finally, we say that an initial density $\bar \rho_0(x,\xi) \geq 0$ on $\mbx \times \fs$ is {\em propagating} if $\overline{V_0} = \fs$ when 
\begin{equation}
V_0 = \{ \xi \in \fs \;|\; \text{$\bar \rho_0(\cdot,\xi)$ is positive on a set of positive measure} \}. \label{V0def}
\end{equation}

\begin{proposition}\label{positive}
	Let $\bar \rho$ be a probability density on $\Pi$ satisfying \eqref{MFE} with initial condition $\bar \rho_0(x,\xi)$. Let $V_0 = V_0(\bar \rho_0)$ be defined by \eqref{V0def}. Then for all $\xi \in \overline{V_0}$, $\inf_{(t,x) \in C} \bar \rho(t,x,\xi) > 0$ for any compact set $C \subset (0,\infty) \times \mathbb{T}^d$. Moreover, for all $\xi \in \fs \setminus \overline{V_0}$, $\bar \rho(t,x,\xi) = 0$ holds for all $t \geq 0$, $x \in \mathbb{T}^d$.
\end{proposition}

\begin{proof}
By Proposition \ref{positive2}, we know that for any compact set $C \subset (0,\infty) \times \mathbb{T}^d$
\begin{equation}
\inf_{(t,x) \in C} \bar \rho(t,x,\xi) > 0 \label{uniflowercomp}
\end{equation}
holds for all $\xi \in V_0$. Now, proceeding inductively, suppose that \eqref{uniflowercomp} holds for all $\xi \in V_n$, for some $n \geq 0$, with $V_n$ defined via \eqref{Vndef}.  Then if $\xi' \in V_{n+1}$, the definition of operator $\bar T^+_R$ in \eqref{TR-def} implies
\[
\inf_{(t,x) \in C}  (\bar T^+_R \bar \rho)(t,x,\xi') > 0
\]
holds for any compact set $C \subset (0,\infty) \times \mathbb{T}^d$. As in \eqref{ulower1}, this implies that for any compact $C' \subset (0,\infty) \times \mathbb{T}^d$, there is a constant $c_1 > 0$ such that
\[
\partial_t \bar \rho(t, x,\xi') \geq \frac{\sigma(\xi')^2}{2} \Delta \bar \rho(t,x,\xi') - \beta_0 \bar \rho(t,x,\xi') + c_1
\]
holds for all $(t,x) \in C'$.  This and the maximum principle implies that the condition \eqref{uniflowercomp} also holds for $\xi'$.  Since \eqref{uniflowercomp} holds for all such $\xi' \in V_{n+1}$, we conclude by induction on $n$ that \eqref{uniflowercomp} holds for all $\xi \in \overline{V_0}$.
\end{proof}

  \bibliographystyle{abbrv}

\bibliography{crn}

\end{document}